\documentclass{amsart}
\pdfoutput = 1
\usepackage{amssymb}
\usepackage[utf8]{inputenc}
\usepackage[dvipsnames]{xcolor}
\usepackage{tikz-cd}
\usepackage{fullpage}
\usepackage{enumitem}

\newtheorem{thm}{Theorem}[section]
\newtheorem{prop}[thm]{Proposition}

\newtheorem{lem}[thm]{Lemma}

\newtheorem{lemma}[thm]{Lemma}
\newtheorem{cor}[thm]{Corollary}

\newtheorem{fact}[thm]{Fact}

\theoremstyle{definition}
\newtheorem{definition}[thm]{Definition}

\newtheorem*{t1}{Theorem \ref{thm:ClosedFormula}}
\newtheorem*{mt2}{Theorem \ref{thm:RespectsOrdering}}
\newtheorem*{mt3}{Theorem \ref{thm:wdvv}}

\newtheorem{assumption}[thm]{Assumption}

\theoremstyle{remark}
\newtheorem{rmk}[thm]{Remark}
\newtheorem{remark}[thm]{Remark}
\newtheorem{ex}[thm]{Example} 

\newtheorem{notation}[thm]{Notation}
\newtheorem{claim}[thm]{Claim}

\newcommand{\trop}{\rm trop}
\DeclareMathOperator{\gl}{gl}

\DeclareMathOperator{\rank}{rank}
\newcommand{\Z}{\mathbb{Z}}
\newcommand{\C}{\mathbb{C}}

\newcommand{\Q}{\mathbb{Q}}

\renewcommand{\P}{\mathbb{P}}
\newcommand{\M}{M}

\newcommand{\abs}[1]{\left\lvert#1\right\rvert}

\newcommand{\Mbar}{\overline{M}}

\title{Genus-zero  $r$-spin theory}
\author[Cavalieri]{Renzo Cavalieri}
\address{
	\begin{tabular}{l}
		Renzo Cavalieri \\ Colorado State University, 1874 Campus Mail, Fort Collins, CO, 80523-1874, USA
	\end{tabular}
}
\email{renzo@math.colostate.edu}

\author[Kelly]{Tyler L. Kelly}
\address{
\begin{tabular}{l}
     Tyler L. Kelly\\
     School of Mathematics, University of Birmingham, Edgbaston, Birmingham B15 2TT, UK
\end{tabular}}
\email{t.kelly.1@bham.ac.uk}

\author[Silversmith]{Rob Silversmith}
\address{
  \begin{tabular}{l}
     Rob Silversmith\\
     Mathematics Institute, University of Warwick, Coventry CV4 7AL, UK
\end{tabular}}
\email{Rob.Silversmith@warwick.ac.uk}

\begin{document}

\begin{abstract}
    We provide an explicit formula for all primary genus-zero $r$-spin invariants. Our formula is piecewise polynomial in the monodromies at each marked point and in $r$. To deduce the structure of these invariants, we use a tropical realization of the corresponding cohomological field theories. We observe that the collection of all WDVV relations is equivalent to the relations deduced from the fact that genus-zero tropical CohFT cycles are balanced. 
\end{abstract}
\maketitle

\setcounter{tocdepth}{1}
\tableofcontents

%%%%%%%%%%%INTRO
\section{Introduction}

This paper achieves two distinct goals. The first is to establish a connection between the combinatorial properties of the genus-zero part of any cohomological field theory (CohFT) and of its tropicalization.
The second is to apply this analysis to the CohFT of $r$-spin Witten classes (see Section \ref{sec:RSpin} for a discussion) and obtain 
a detailed understanding of the corresponding numerical invariants. We first discuss this second goal, which allows for simple and concrete statements, and then enlarge the scope to discuss the more general combinatorial journey that led to these results.

A {\it genus-zero $r$-spin invariant} $w_r(\vec m)$, often denoted $\langle \tau^{m_1} \cdots \tau^{m_n}\rangle$ in the literature, is an intersection number on a moduli stack of $n$-pointed, genus-zero $r$-spin curves obtained by integrating the Euler class of the Witten bundle over a connected component of the moduli stack. Here $r$ is a positive integer, and $\vec m = (m_i)\in \{ 1, 2, \ldots , r\}^n$, with $\sum_{i=1}^nm_i = (n-2)(r+1)$, is called a {\it numerical monodromy vector}. These intersection numbers have a rich tradition: originally proposed by Witten \cite{Witten}, they were developed for many reasons including their CohFT structure \cite{JKV01}, their setting in mirror symmetry as an enumerative theory for the Landau-Ginzburg model $(\mathbb{C}, \mu_r, x^r)$ \cite{FJR11}, and their applications towards tautological relations \cite{PPZ2019}. Our first main result gives a closed formula for (primary) genus-zero $r$-spin invariants.

\begin{t1}
The genus-zero $r$-spin invariant of a numerical monodromy vector $\vec m=(m_1,\ldots,m_n)$ is
    \begin{align}\label{eq:ClosedFormulai}
        w_r(\vec m)=\frac{1}{2\cdot r^{n-3}}\sum_{\substack{S\subseteq[n]\\\sum_{i\in S}m_i\ge(\abs{S}-1)r+n-2}}(-1)^{1+\abs{S}}\prod_{k=1}^{n-3}\Bigl(\Bigl(\sum_{i\in S}m_i\Bigr)-(\abs{S}-1)r-k\Bigr).
    \end{align}
\end{t1}

Viewing $ m_1,\ldots,m_n,r$  as variables, one may think of $w_r(\vec m)$ as a function defined on the integral lattice points of an unbounded polyhedron $\mathbf{M}_n$ in $\mathbb{R}^{n}\times \mathbb{R}$.  Formula \eqref{eq:ClosedFormulai} implies that $w_r(\vec m)$ is a piecewise-polynomial function  of degree $n-3$. The chambers of polynomiality are polyhedra that overlap in affine linear strips of width 
$n-4$. The precise statements and the equations for the walls are given in {Corollary \ref{cor:PiecewisePolynomial}}. Previously, it was known that $w_r(\vec m)$ is a piecewise polynomial when $n=3,4$ \cite[Prop. 6.1]{JKV01}. It was also known that $w_r(\vec m)$ is equal to the dimension of $\mathfrak{sl}_2(\C)$-invariant subspace of a certain tensor product of symmetric powers of the standard representation of $\mathfrak{sl}_2(\C)$ \cite[Thm. 2]{PPZ2019}, and it is the genus-zero part of the r-KdV hierarchy \cite{JKV01}.

We also unveil properties of genus-zero $r$-spin invariants which are not apparent from their expression in \eqref{eq:ClosedFormulai}. For a fixed value of $r$, the {\it dominance order} is a partial ordering on the monodromy vectors $\vec m$ that, roughly speaking, says they become smaller as they get closer to the small diagonal of $\mathbb{R}^n$---equivalently, as the entries of $\vec m$ get more equidistributed (see Section~\ref{sec:monotonicity} for a precise definition). We show that genus-zero $r$-spin invariants are monotonic with respect to this partial ordering.

\begin{mt2}[Part (1)]
For any fixed $r,n$, the function $w_r(\vec m)$ is weakly order-reversing with respect to the dominance ordering.   That is, 
\begin{equation} \label{eq:monot}
  \vec m\le\vec m'   \ \ \ \implies  \ \ \ w_r(\vec m)\ge w_r(\vec m').
\end{equation}
\end{mt2}

In \cite[Prop. 1.4]{PPZ2019} the authors use character theory of $\mathfrak{sl}_2(\C)$ to show that for a monodromy vector of length $n$, $w_r(\vec m) = 0$ if $m_i\leq n-3$ for any of the entries of $\vec m$; we give a direct combinatorial proof of this fact. This vanishing, together with Theorem \ref{thm:RespectsOrdering} Part (1),  implies the following positivity statement.

\begin{mt2}[Part (2)]
For any $r$ and $\vec m$, we have
 $w_r(\vec m)\ge 0$, with
 $w_r(\vec m)\not =0$ if and only if $n-2\le m_i \le r-1$ for all $i$.
\end{mt2}

\noindent Notably, the formula for genus-zero $r$-spin invariants is more complicated than the closed formula for \emph{open} $r$-spin invariants found in \cite[Theorem 1.2]{BCT-OpenII}.

All these structure results about genus-zero $r$-spin invariants follow from a collection of linear recursive relations, proved in Theorem \ref{thm:Recursion}. Such relations may be derived in two equivalent ways, which leads us to discussing the first stated goal of the paper.

Genus-zero $r$-spin invariants are degrees of the zero-dimensional cycles for the CohFT of Witten classes, and comprise the \emph{numerical part} of the CohFT. Very roughly speaking, a CohFT (see \cite{PandhaCohCa} for precise definitions) is an infinite collection of Chow classes on moduli spaces of curves that are self-referential with respect to restriction to boundary strata; this means that the intersection of a CohFT class with a stratum $$
\Delta = \gl_\ast\left( \prod_i \overline{M}_{g_i, n_i}\right)$$ is equal to the pushforward via the corresponding gluing morphism of (a linear combination of products of pullbacks via the coordinate projections of) CohFT classes from the factors.
The basic linear equivalence among the three boundary points of $\overline{M}_{0,4}\cong \P^1$ may be pulled-back via forgetful morphisms and/or pushed-forward via gluing morphisms, then intersected with a CohFT class and integrated to obtain a collection of relations among the numerical CohFT classes. Such relations are called {\it WDVV relations} and are familiar tools in and around Gromov-Witten theory, where they are, for example, responsible for the associativity of the quantum product \cite{FuPa}. One way to obtain the recursions in Theorem \ref{thm:Recursion} is via WDVV. Another way is to study the {\it tropicalization} of the CohFT. We summarize here the discussion in Section \ref{sec:TropCohFT}.

Tropical geometry \cite{BIMS,MikhalkinICM} provides a combinatorialization of standard algebraic geometry concepts, and in the last few decades it has found interesting applications to enumerative geometry (e.g. \cite{MikhalkinEnumerative,GMcapoharris, GathmannKerberMarkwig}) and tautological intersection theory of moduli spaces of curves (e.g. \cite{MikhalkinModuli,RauModuli,KerberMarkwig2009,Gross, tropicalpsi}). The moduli space of tropical curves $M_{g,n}^{\trop}$ \cite{ACP, CCUW} may be identified with the boundary complex of $\overline{M}_{g,n}$. For any Chow class $\alpha \in A^k(\overline{M}_{g,n})$, its tropicalization $\alpha^{\trop}$ is a codimension-$k$ weighted cone subcomplex of $M_{g,n}^{\trop}$ where each codimension-$k$ cone of $M_{g,n}^{\trop}$ is weighted by the intersection number of $\alpha$ with the corresponding $k$-dimensional stratum of $\overline{M}_{g,n}$. If $\Omega$ is a CohFT class, then it follows from the strata restriction properties that the coefficients of $\Omega^{\trop}$ are functions of numerical CohFT invariants. In other words, the tropicalization of a CohFT is completely controlled by its numerical part (Proposition \ref{thm:TropicalCycle}).

In genus zero, the moduli space of curves is a {\it tropical compactification} \cite{Tev, GM}, meaning that $\overline{M}_{0,n}$ admits an embedding into a (non-proper) toric variety $X_\Sigma$ such that the boundary stratification of $\overline{M}_{0,n}$ coincides with the restriction of the toric boundary of $X_\Sigma$. It follows that $M_{0,n}^{\trop}$ may be identified with $\Sigma$, and  therefore viewed as a {\it balanced fan} inside the vector space $N_\mathbb{R}$, spanned by the cocharacter lattice of the torus of $X_\Sigma$. There is a natural isomorphism $\varphi: A^\ast(\overline{M}_{0,n}){\to} A^\ast(X_{\Sigma})$ 
and the tropicalization $\alpha^{\trop}$ of a cycle $\alpha\in A^\ast(\overline{M}_{0,n})$ coincides with the Minkowski weight presentation of $\varphi(\alpha)$ \cite{FuStu,Katz}.

This in turn implies that $\alpha^{\trop}$ satisfies the {\it balancing condition} \eqref{eq:balancingcondition}, which translates into a collection of linear equations on the weights of the cones of $\alpha^{\trop}$. When $\alpha$ is a CohFT class, one then obtains equations among the numerical CohFT invariants. Our next result relates these equations to the relations obtained from WDVV.  

\begin{mt3}
The collection of balancing equations for all tropical cycles $\Omega^{\trop}$ imposes the same constraints as all WDVV relations on the numerical invariants of $\Omega$.
\end{mt3}

The statement of this theorem can be made precise (but probably unnecessarily confusing) by saying that the two sets of equations cut down the same subvariety in some countable dimensional affine space coordinatized by the discrete invariants indexing the classes of \emph{arbitrary} CohFTs. It could also be explained in working terms: if one is trying to reconstruct numerical CohFT invariants recursively given some initial conditions, it is equivalent to use the set of relations coming from WDVV or the balancing equations. 
The proof of Theorem \ref{thm:wdvv} shows just  how tight the connection is: balancing at the cone point for one-dimensional tropical cycles corresponds to WDVV relations obtained by pull-back via forgetful morphisms, and balancing along faces of higher dimensional cycles correspond to WDVV relations pushed-forward via appropriate gluing morphisms.

The tropicalization of the moduli space of $r$-spin curves has been studied in \cite{CMP20, AbreuPaciniSecco}, where the authors identify a skeleton of the Berkovich analytification of algebraic $r$-spin curves with a cone complex parameterizing tropical $r$-spin curves. Our perspective is to directly  tropicalize the cycles of the $r$-spin CohFT, bypassing the need for a tropical version of the theory of Witten classes on moduli spaces of tropical $r$-spin curves. Once a theory of tropical vector bundles and Chern classes is established, it would be natural to try to exhibit the tropicalization of a genus-zero $r$-spin Witten class as the Euler class of a tropical version of the Witten bundle.  

The paper is written with the intention of being accessible to readers from different mathematical communities, in order to stimulate hopefully productive interactions. Section \ref{sec:RSpin} provides background and intuition on the CohFT of $r$-spin Witten classes, aimed at readers from tropical enumerative geometry. Section \ref{sec:TropCohFT} gives background on tropical intersection theory, aimed at algebraic geometers interested in intersection theory on moduli spaces.
Section \ref{sec:recursivestructure} derives recursions among genus-zero $r$-spin numerical invariants, and these recursions are used in Section \ref{sec:ClosedFormula} to compute the closed formula \eqref{eq:ClosedFormulai}.
In Section \ref{sec:monotonicity} we study the monotonicity of genus-zero $r$-spin invariants with respect to the dominance ordering.

\subsection*{Acknowledgements}
The authors would like to thank Alexander Barvinok, Alessandro Chiodo,  Hannah Markwig and Diane Maclagan for discussions relating to this work. We also thank the referee for their detailed comments and feedback that have improved the paper. The first author is grateful for support from Simons Collaboration Grant 420720 and  NSF grant DMS-2100962.
The second author acknowledges support provided by the EPSRC under grant EP/N004922/2, the UK Research and Innovation (UKRI) Talent and Research Stabilisation Fund, and the UKRI Future Leaders Fellowship MR/T01783X/1. The second and third authors would also like to thank the hospitality of Colorado State University during their visit. 
Lastly, all the authors would like to thank the Instituto Nacional de Matemática Pura e Aplicada (IMPA) for their hospitality; we discovered Theorem~\ref{thm:ClosedFormula} while participating in the conference ``ALGA XV.'' 

%%%%%%%%%%% ALGEBRAIC BACKGROUND
\section{The $r$-spin cohomological field theory}\label{sec:RSpin}
We are interested in computing $r$-spin invariants, which are defined as intersection numbers over moduli spaces of roots of line bundles on curves.  The literature on $r$-spin curves is both technical and full of conflicting conventions, hence we provide in this section a self-contained overview of this story. Since our focus is combinatorial, we sweep technicalities under the rug when possible. 

\subsection{Moduli of $r$-spin curves}
 Fix an integer $r\ge 2$. Let $(C, p_1, \dots, p_n)\in\M_{g,n}$ be a smooth $n$-marked curve of genus $g$, and fix $m_1, \dots, m_n \in \{1, \dots, r\}$ such that
\begin{equation}\label{divisibility}
2g - 2 + n - \sum_{i=1}^n m_i \in r\mathbb{Z}.
\end{equation} An \emph{$r$-spin structure of type $\vec{m} := (m_1, \dots, m_n)$} on $C$ is a line bundle $L$ on $C$ together with an isomorphism
$$
\varphi: L^{\otimes r} \stackrel{\sim}{\longrightarrow} \omega_{C, \textup{log}}\left(-\sum_{i=1}^n m_i p_i\right)
$$
where $\omega_{C, \textup{log}}=\omega_C(\sum_{i=1}^n p_i)$ is the log canonical bundle 
of $C$. We  refer to $m_i$ as the \emph{monodromy} of $L$ at $p_i$; the terminology arises from an equivalent formulation of $r$-spin curves in which $p_1,\ldots,p_n$ are orbifold points of $C$, and $m_i$ is the monodromy of $L$ at $p_i$. We call $\vec m$ (subject to the condition \eqref{divisibility}) a \emph{monodromy vector}, and call the data $(C, p_i, L, \varphi)$ a (smooth) \emph{$r$-spin curve}.

\begin{remark}\label{rem:offbyone}
The integers $m_i$ are off by one from some standard references on $r$-spin curves (e.g., \cite{JKV01,BCT-Extended,PPZ2019}), but match others (e.g., \cite{Chi08a,FJR11, FJR13}).
\end{remark}

For an $r$-spin curve of type $\vec m$, by the existence of $\varphi$, we immediately see
\begin{equation}\label{deg L}
\deg L = \frac{2g - 2 + n - \sum_{i=1}^n m_i}{r} \in \Z.
\end{equation}

We denote by $\M_{g,n}^{r}(\vec{m})$ the moduli stack of smooth $r$-spin curves of type $\vec{m}$, and let $\M_{g,n}^r=\bigsqcup_{\vec m}\M_{g,n}^{r}(\vec{m})$. There is a natural compactification $\Mbar_{g,n}^r=\bigsqcup_{\vec m}\Mbar_{g,n}^{r}(\vec{m})$ due to Abramovich-Jarvis and Chiodo \cite{AJ03, Chi08}---though other compactifications exist \cite{Jar98, Jar00, CCC07}. In the language of \cite{Chi08}, the stack $\Mbar_{g,n}^r(\vec m)$ is the moduli stack of \emph{stable $r$-spin curves}, i.e. $r$-th roots $L$ of $\omega_{C,\log}$ on nodal orbicurves $C$ with $\Z/r\Z$-orbifold structure at marked points and nodes (and nowhere else), where $m_i$ is the monodromy of $L$ at $p_i$. This space admits a (finite flat surjective) forgetful map $\rho:\Mbar_{g,n}^{r}(\vec{m}) \to \Mbar_{g,n}$, see \cite[Thm. 4.2.3]{Chi08}. 
\begin{remark}\label{rem:GenusZeroForget}
    When $g=0$ (and necessarily $n\ge3)$, the forgetful map $\rho$ is a bijection. It is not an isomorphism for stack-theoretic reasons---due to the fact that $r$-spin curves admit extra automorphisms---but we can largely ignore this subtlety for our purposes. 
\end{remark} 

An important aspect of $\Mbar_{g,n}^r(\vec m)$---similarly to $\Mbar_{g,n}$---is its recursive structure, which we briefly describe. 
On a stable $r$-spin curve $(C,p_i,L,\varphi)$, each node $\eta$ of $C$ locally looks like the quotient of an ordinary node $V(xy)$ by an action of $\Z/r\Z$ via $$(x,y) \mapsto (\zeta x, \zeta^{-1}y)$$ where $\zeta$ is an $r$-th root of unity. Under this identification, $L$ is locally the quotient of $V(xy)\times\C$ by $\Z/r\Z$, acting by $$(x,y,t)\mapsto(\zeta x, \zeta^{-1}y,\zeta^a t)$$ for some $a\in\{1,\ldots,r\}.$ 
Exchanging the roles of $x$ and $y$ would instead give an action $t \mapsto \zeta^b t$ where $a+ b = 0 \pmod r$. We thus have a well-defined notion of the monodromy at $\eta$, \emph{after picking a branch of the node}. 

\begin{ex}\label{ex: codim 1 stratum}
Consider a genus-$g$ stable $r$-spin curve $(C, p_1, \dots, p_n, L, \varphi)$ of type $(m_1, \dots, m_n)$, with exactly two irreducible components $C_1$ and $C_2$ of genera $g_1$ and $g_2$, joined at a single node $\eta$. Denote by $J\subseteq[n]$ the set of marked points on $C_1$ (so $J^c$ is the set of marked points on $C_2$). We additionally mark the two preimages $p_J\in C_1$ and $p_{J^c}\in C_2$ of $\eta$ under the normalization map, and pull back $L$ to each component. This yields two smooth $r$-spin curves $$(C_1,\{p_i\}_{i\in J}\cup\{p_J\},L,\varphi)\quad\quad\text{and}\quad\quad(C_2,\{p_i\}_{i\in J^c}\cup\{P_{J^c}\},L,\varphi)$$ with monodromies $m_J,m_{J^c}\in\{1,\ldots,r\}$ at $p_J$ and $p_{J^c}$ respectively, where\footnote{There is a small subtlety here involving the dimension of the Witten class when $m_J=r$, which we may ignore by Proposition \ref{prop:RamondVanishing} below. See \cite[Sec. 3.2]{BCT-Extended} for details.}
\begin{equation}\label{monodromy at nodes higher genus}
    m_J \equiv 2g -2+(|J|+1) - \sum_{i\in J} m_i \pmod r \quad\text{ and }\quad m_{J^c} \equiv 2g - 2 + (|J^c|+1) - \sum_{i\in J^c} m_i \pmod r.
\end{equation}
\end{ex}
Example \ref{ex: codim 1 stratum} illustrates the \emph{recursive} 
boundary stratification  of $\Mbar_{g,n}^r(\vec m)$, which compatibly matches the well-known recursive boundary stratification of $\Mbar_{g,n}$. If $D_{g_1,g_2,J}^r\subseteq\Mbar_{g,n}^r(\vec m)$ denotes 
the divisor of stable $r$-spin curves generically of the type in the example, we have a bijection $$D_{g_1,g_2,J}^r\to\Mbar_{g_1,\abs{J}+1}^r(\{m_i\}_{i\in J}\cup\{m_J\})\times\Mbar_{g_2,\abs{J^c}+1}^r(\{m_i\}_{i\not\in J}\cup\{m_{J^c}\}).$$ The bijection is compatible with the decomposition of the corresponding boundary divisor $D_J\subseteq\Mbar_{g,n}$ as a product\footnote{If $g_1=g_2$ and $n=0$ one must take a quotient of the product, and the same is true for $r$-spin curves above.} $D_J\cong\Mbar_{g_1,\abs{J}+1}\times\Mbar_{g_2,\abs{J^c}+1}.$ (The bijection is not  an isomorphism because the stack structure on the two sides is different.)
\begin{ex}\label{ex: another codimension 1 stratum}
    We may similarly consider the locus in $\Mbar_{g,n}^r(\vec m)$ associated to any fixed combinatorial type of curve (i.e. fixed \emph{dual graph} $\Gamma$). We may carry out a similar analysis to that in Example \ref{ex: codim 1 stratum}, again via the normalization map, with one caveat, illustrated in an example as follows. If $D^{\mathrm{loop}}\subseteq\Mbar_{g,n}^r(\vec m)$ denotes  the closure of the locus of stable $r$-spin curves $C$ with one irreducible component with a self-node $\eta$ (so $\Gamma$ has a single vertex with a self-loop), then the monodromies at the branches of $\eta$ are not uniquely determined; we instead get a bijection $$D^{\mathrm{loop}}\to\bigsqcup_{k=1}^r\Mbar_{g-1,n+2}^r(\vec m\cup\{k,r-k\}).$$
\end{ex}
Adapting the ideas from  Examples \ref{ex: codim 1 stratum} and \ref{ex: another codimension 1 stratum}, one can characterize any stratum in $\Mbar_{g,n}^r(\vec m)$ recursively.

\subsection{Genus-zero Witten classes and $r$-spin invariants}

In this section we describe the natural class of intersection numbers on $\Mbar_{g,n}^r(\vec m)$ we are interested in,  restricting to the case $g=0$. 
The \emph{Witten bundle} is the derived pushforward
$$
\mathcal{W}_{0,n}^r = (\mathbf{R}^1 \pi_* \mathcal{L}_{0,n}^r)^\vee,
$$ where $\mathcal{L}_{0,n}^r$ is the universal line bundle on the universal curve $\mathcal{C}_{0,n}^r \stackrel{\pi}{\rightarrow} \Mbar_{0,n}^r$.
We denote by $\mathcal{W}_{0,n}^r(\vec{m})$ the restriction of the Witten bundle to $\Mbar_{0,n}^r(\vec{m})\subseteq\Mbar_{0,n}^r$.

A standard argument (see \cite[Prop. 4.4]{JKV01}) shows that $\mathbf{R}^0 \pi_* \mathcal{L}_{0,n}^r=0$. This fact depends crucially on having $m_i\ge1$ for all $i$. Applying Riemann-Roch and~\eqref{deg L} yields
\begin{align}\label{eq:RankOfWittenBundle}
\rank\mathcal{W}_{0,n}^r(\vec{m}) = -1 +\frac{ 2 - n + \sum_i m_i}{r}.
\end{align}
\begin{definition}\label{def:Witten class}
    The (genus-zero) $r$-spin \emph{Witten class} with monodromy vector $\vec m$ is $$W_r(\vec m)=r\cdot\rho_*\left(e(\mathcal{W}_{0,n}^{r}(\vec m))\right)\in A^*(\Mbar_{0,n}),$$ where $e$ denotes the Euler class, and $\rho:\Mbar_{0,n}^r(\vec m)\to\Mbar_{0,n}$ is the (bijective, by Remark \ref{rem:GenusZeroForget}) forgetful map. 
    The (primary, genus-zero) \emph{$r$-spin invariant} associated to the monodromy vector $\vec{m}=(m_1, \dots, m_n)$ is
$$
w_r(\vec m):= \int_{[\Mbar_{0,n}]}W_r(\vec m)=r\cdot\int_{[\Mbar_{0,n}^r(\vec m)]}e(\mathcal{W}_{0,n}^{r}(\vec m)).
$$ 

\end{definition}

In the literature, $w_r(\vec m)$ is often denoted $\langle \tau^{m_1} \cdots \tau^{m_n} \rangle$. Note that $w_r(\vec m)\ne0$ only if $\textrm{rank}\ \mathcal{W}_{0,n}^r(\vec{m})=\dim\Mbar_{0,n}^r(\vec m)=n-3$, or equivalently, only if
\begin{equation}\label{eqn: dim rank match}
\sum_i m_i = (n-2)(r+1).
\end{equation}
In this case we say $\vec m$ is \emph{numerical}. This is the case when $W_r(\vec m)$ is a zero-dimensional cycle on $\Mbar_{0,n}^r(\vec w)$; in general, $W_r(\vec m)$ is a cycle of dimension $\frac{1}{r}((n-2)(r+1)-\sum_im_i$).

\begin{rmk}
    The factor of $r$ in Definition \ref{def:Witten class} is essentially a matter of convention---see \cite[Cor. 3.9]{JKV01} and the succeeding discussion.
\end{rmk}

\begin{rmk}\label{rmk: rspinhistory}
    The Witten bundle, Witten class, and $r$-spin invariants may also be defined in higher genus; we briefly mention some of the history. Jarvis, Kimura and Vaintrob first gave a collection of axioms for---but not a construction of---a \emph{virtual $r$-spin class}, as a collection of cohomology classes on each boundary stratum of each moduli space $\Mbar_{g,n}^r(\vec m)$ of $r$-spin curves \cite[Def. 4.1]{JKV01}. They proved that any virtual $r$-spin class defines a \emph{cohomological field theory (CohFT)} (see Section \ref{sec:RSpinProperties}) by pushing forward to $\Mbar_{g,n}$, and also that 
    any virtual $r$-spin class agrees with the Witten class  defined above in genus zero \cite[Rem. 4.2.4]{JKV01}. 
    
    In \cite[Sec. 1.3]{Witten}, Witten had earlier outlined an analytic construction of $r$-spin invariants. 
    Fan, Jarvis, and Ruan \cite{FJR11,FJR13} carried out and generalized this construction, and proved it defines a virtual $r$-spin class in the sense of \cite{JKV01}. 
    As desired in \cite[Rem. 4.2.5]{JKV01}, an algebraic construction soon followed, constructed by Polishchuk and Vaintrob \cite{PV11} using matrix factorizations, which has been generalised to enumerative theories associated to general gauged linear sigma models \cite{CFGKS18, FK20}.
\end{rmk}

\subsection{Properties of $r$-spin classes and numerical invariants}\label{sec:RSpinProperties}
We now list the properties of the classes $W_r(\vec m)$ that we will need. Propositions \ref{ax:RestrictToStratum}--\ref{prop:34Point} come directly from \cite{JKV01}.

Fix a monodromy vector $\vec m.$ Recalling Example \ref{ex: codim 1 stratum}, let $D_J\subseteq\Mbar_{0,n}$ be the closure of the locus of curves with two components $C_1$ and $C_2$, connected at a node $\eta$, with $C_1$ containing the marks in $J$. As before, let $p_J\in C_1$ and $p_{J^c}\in C_2$ be the preimages of $\eta$ under normalization, and define $m_J,m_{J^c}\in\{1,\ldots,r\}$ by:
\begin{align}\label{monodromy at nodes}
m_J&\equiv(\abs{J}-1)(r+1)-\sum_{i\in J}m_i\pmod r&m_{J^c}&\equiv(\abs{J^c}-1)(r+1)-\sum_{i\in J^c}m_i\pmod r.
\end{align}
Note that \eqref{monodromy at nodes} is equivalent to \eqref{monodromy at nodes higher genus} when $g=0$. The particular form of \eqref{monodromy at nodes} will be convenient, see Proposition \ref{prop:1DVanish}.

\begin{prop}[{\cite[Axiom C2, Cor. 3.9]{JKV01}}]\label{ax:RestrictToStratum}
We have $$W_r(\vec m)|_{D_J}=W_r((m_i)_{i\in J}\cup\{m_J\})\boxtimes W_r((m_i)_{i\in J^c}\cup\{m_{J^c}\}),$$ where $\boxtimes$ denotes the product of pullbacks from the factors of $D_J\cong\Mbar_{0,\abs{J}+1}\times\Mbar_{0,\abs{J^c}+1}$.
\end{prop}
\begin{remark}
    In genus zero, Proposition \ref{ax:RestrictToStratum} is essentially the defining property of a cohomological field theory, so we'll refer to it as the \emph{CohFT property}. See Section \ref{sec:TropCohFT} for further discussion.
\end{remark}

\begin{prop}[{\cite[Axiom 4]{JKV01}}]\label{prop:RamondVanishing}
    If $\vec m$ is a monodromy vector and $m_i = r$ for some $i$, then $W_r(\vec m) = 0$.
\end{prop}

\begin{remark}
    In the physics literature, a marked point $p_i$ is referred to as \emph{Ramond} if $m_i=r$ (and  \emph{Neveu-Schwarz} otherwise), hence Proposition~\ref{prop:RamondVanishing} is often referred to as \emph{Ramond vanishing}.
\end{remark}
The following follows from \cite[Axiom 5]{JKV01}.
\begin{prop}\label{prop:Insertion1Vanish}
   Let $\vec m$ be a numerical monodromy vector of length $n\ge 4$ with $m_i =1$ for some $i$. Then $w_r(\vec m) = 0$.
\end{prop}

\begin{prop}[{\cite[Prop. 6.1]{JKV01}}]\label{prop:34Point} The 3-point and 4-point $r$-spin invariants are as follows:
\begin{enumerate}
    \item For any numerical $r$-spin monodromy vector $\vec m=(m_1,m_2,m_3)$, we have $w_r(\vec m)=1.$
    \item \label{item:4pt} For any numerical $r$-spin monodromy vector $\vec m=(m_1,m_2,m_3,m_4)$, we have $$w_r(m_1,m_2,m_3,m_4)=
\frac{1}{r}\min(m_1-1,m_2-1,m_3-1,m_4-1,r-m_1,r-m_2,r-m_3,r-m_4).$$
\end{enumerate}
\end{prop}

\begin{prop}\label{prop:1DVanish}
Let $\vec m=(m_1,\ldots,m_n)$ be an $r$-spin monodromy vector with $n>4$ such that $W_r(\vec m)$ is a 1-dimensional cycle, i.e. $\sum_im_i=(n-2)(r+1)-r$. Then for $J\subseteq\{1,\ldots,n\}$ with $2\le\abs{J}\le n-2$, the restriction $W_r(\vec m)|_{D_J}$ is nonzero only if the following two inequalities hold:
$$
\sum_{i\in J} m_i < (|J|-1)(r+1) \quad\quad\text{ and } \quad\quad\sum_{i\in J^c} m_i < (|J^c|-1)(r+1).
$$
\end{prop}
\begin{proof}
Let $m_J,m_{J^c}$ be as in \eqref{monodromy at nodes}. Without loss of generality, suppose that 
$
\sum_{i\in J} m_i \ge (|J|-1)(r+1),
$
so in particular $$m_J+\sum_{i\in J} m_i > (|J|-1)(r+1).$$ Then by \eqref{eq:RankOfWittenBundle}, we have $$\rank{\mathcal W}_r((m_i)_{i\in J}\cup\{m_J\})>\abs{J}-2=\dim\Mbar_{0,\abs{J}+1},$$ and so $W_r((m_i)_{i\in J}\cup\{m_J\})=0$ for dimension reasons. By Proposition~\ref{ax:RestrictToStratum}, $W_r(\vec m)|_{D_J} = 0$.
\end{proof}

\section{Tropical realizations of classes from cohomological field theories.}\label{sec:TropCohFT}

In this section we recall some notions about the tropical intersection theory of $M_{0,n}^{\trop}$ and its relation to the  intersection theory of $\overline{M}_{0,n}$. We define the tropicalization of a CohFT, then restrict our attention to genus zero and describe the relationship between the balancing condition for tropicalized CohFT cycles and the WDVV relations. Basics of tropical moduli spaces that are assumed here may be found in \cite{SpeyerSturmfels2004,MikhalkinModuli,GathmannKerberMarkwig, ACP}.

\subsection{Tropicalization of Cycle Classes in $\overline{M}_{0,n}$} 

\label{sec:tctc}

The moduli space $\overline{M}_{0,n}$ of rational, stable, $n$-pointed curves is a tropical compactification of $M_{0,n}$ \cite{Tev}: the noncompact space $M_{0,n}$ may be realized as a closed subvariety of a torus $T$. Denote the cocharacter lattice of $T$ by $N_T$. The tropicalization of $M_{0,n}$ lives inside the vector space $Q_{[n]} := N_T\otimes_{\mathbb{Z}}{\mathbb{R}}$ as a balanced fan, which is naturally identified with
$M_{0,n}^{\trop}$, the moduli space of rational, stable, tropical $n$-pointed curves (Fact \ref{fact:BalancedFan}).
Given a Chow class in $\overline{M}_{0,n}$, one may define (Definition \ref{def:tropcycles}) its \emph{tropicalization}, a weighted subfan of $M_{0,n}^{\trop}$ that satisfies the balancing condition. We review the aspects of this story that we will be using, and refer the reader to \cite{KerberMarkwig2009,Katz} for proofs. Some extended examples are included for the benefit of the unfamiliar readers.

 \subsubsection*{The ambient space} The vector space $Q_{[n]}$ has dimension $\binom{n-1}{2}-1$; we describe it by giving it a (non-canonical) set of generators and relations. Fix $k\in [n]$.  Then the set of pairs $\{i,j\}$ with $i,j\not=k$  gives a system of $\binom{n-1}{2}$ vectors $\mathbf{r}_{\{i,j\}}^k$ generating $Q_{[n]}$. Their $\mathbb{Z}$-span gives the integral lattice  in $Q_{[n]}$ and they are subject to the unique relation (up to scaling): \begin{equation}\label{eq:uniquerelations}
    \sum_{\{i,j\} \in [n]\setminus\{k\}} \mathbf{r}_{\{i,j\}}^k = 0 \in Q_{[n]}. 
 \end{equation} In other words, the vectors $\mathbf{r}_{\{i,j\}}^k$ may be thought of as the the primitive vectors generating the rays for the fan of projective space $\mathbb{P}^{\binom{n-1}{2}-1}$.
\subsubsection*{The embedding of $M_{0,n}^{\trop}$ in $Q_{[n]}$} The space
$M_{0,n}^{\trop}$ is a cone complex parameterizing stable metric trees with ``legs'' labeled by $[n]$. (A \emph{leg} is a labeled half-edge incident to a vertex, usually taken to have infinite length.) For each topological type of tree $\Gamma$, one has a cone isomorphic to $\mathbb{R}_{\ge 0}^{|E(\Gamma)|}$ corresponding to all possible edge lengths, and the cones are glued together along faces by declaring a graph with an edge of length zero equivalent to the graph obtained by contracting of that edge. The abstract cone complex $M_{0,n}^{\trop}$ admits a natural embedding into $Q_{[n]}$, which we now describe. We first give the image ${\mathbf{v}_I}\in Q_{[n]}$ of primitive vectors  generating the rays of $M_{0,n}^{\trop}$; these correspond to trees with a single edge of length one, separating the set of legs into two parts $I, I^c$ both of size at least $2$. Without loss of generality, we assume that $k\not \in I$. We declare
\begin{equation} \label{eq:relvr}
    {\mathbf{v}_I} = \sum_{\{i,j\}\subseteq I} \mathbf{r}_{\{i,j\}}^k.
\end{equation}
Given this information, the image of $M_{0,n}^{\trop}$ in $Q_{[n]}$  is determined by multi-linearity. A point ${\bf x}\in Q_{[n]}$ of (the image of) 
 $M_{0,n}^{\trop}$ corresponds to an $n$-marked metric tree $\Gamma$. Each edge $e\in E(\Gamma)$ produces a two-part partition $I_e, I_e^c$ of the set of indices, where we again assume $k\not\in I_e$. If $l_e$ denotes the length of the edge $e$, then 
\begin{equation}
    {\bf x} = \sum_{e\in E} l_e \mathbf{v}_{I_e}.
\end{equation}

\begin{rmk}
    Both the vector space $Q_{[n]}$ and the embedding of $ M_{0,n}^{\trop}$ may be described in a canonical way (i.e. without choosing a distinguished index $k$), see \cite[Section 2]{KerberMarkwig2009}. The choice of $k$ provides a non-canonical system of generators that makes both statements and computations more concrete.
\end{rmk}
\subsubsection*{Balancing} A pure $d$-dimensional weighted fan $\Sigma$ in a vector space $Q$ with an integral structure is called \emph{balanced} if for every face $\tau$ of dimension $d-1$, one has
\begin{equation}
  \label{eq:balancingcondition}  \sum_{\sigma\succ \tau} w(\sigma) \ \mathbf{u}_{\tau/\sigma} = 0 \in Q/\langle \tau \rangle_{\mathbb{R}},
\end{equation}
where the sum ranges over all $d$-dimensional cones $\sigma$ of $\Sigma$ that contain $\tau$ as a face, $w(\sigma)$ is the weight of the cone $\sigma$,  and $\mathbf{u}_{\tau/\sigma}$ is a primitive normal vector to $\tau$ in $\sigma$, i.e. a primitive vector in the quotient space $\langle\sigma\rangle_\mathbb{R}/\langle\tau\rangle_\mathbb{R} \subseteq Q/\langle\tau\rangle_\mathbb{R}$.

\begin{fact}[For a proof, see {\cite[Thm. 3.7]{GathmannKerberMarkwig}}]\label{fact:BalancedFan} 
The space $M_{0,n}^{\trop}\subseteq Q_{[n]}$, with the weight function $w(\sigma) = 1$ for all top dimensional cones $\sigma$, is a balanced fan of pure dimension $n-3$.
\end{fact}

\begin{ex} We illustrate the ideas presented in the simplest nontrivial example, $M_{0,4}^{\trop}$, drawn in Figure \ref{fig:m04}.
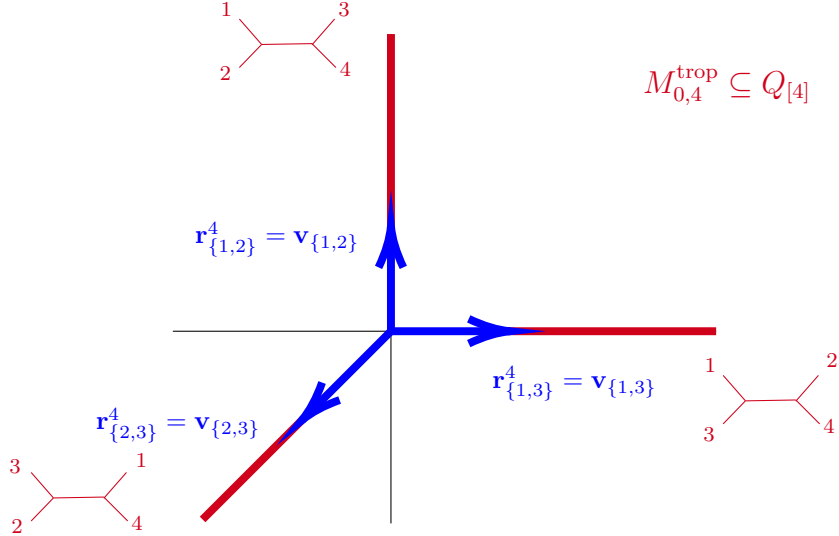
\begin{figure}[tb]
    \centering
\begin{tikzpicture}[x=0.75pt,y=0.75pt,yscale=-1,xscale=1]

%Straight Lines [id:da9140442586528137] 
\draw    (200,287) -- (200,40) ;
%Straight Lines [id:da9540771669771588] 
\draw    (90,190) -- (364,190) ;
%Straight Lines [id:da5875217544731819] 
\draw [color={rgb, 255:red, 208; green, 2; blue, 27 }  ,draw opacity=1 ][line width=3]    (260,190) -- (364,190) ;
%Straight Lines [id:da40672731052867084] 
\draw [color={rgb, 255:red, 208; green, 2; blue, 27 }  ,draw opacity=1 ][line width=3]    (200,137) -- (200,40) ;
%Straight Lines [id:da2571164485999706] 
\draw [color={rgb, 255:red, 208; green, 2; blue, 27 }  ,draw opacity=1 ][line width=3]    (155,235) -- (105,285) ;
%Straight Lines [id:da24666385006201663] 
\draw [color={blue}  ,draw opacity=1 ][line width=3]    (200,190) -- (260,190) ;
\draw [shift={(260,190)}, rotate = 180] [color={blue}  ,draw opacity=1 ][line width=3]    (20.77,-6.25) .. controls (13.2,-2.65) and (6.28,-0.57) .. (0,0) .. controls (6.28,0.57) and (13.2,2.65) .. (20.77,6.25)   ;
%Straight Lines [id:da399040438574497] 
\draw [color={blue}  ,draw opacity=1 ][line width=3]    (200,190) -- (200,137) ;
\draw [shift={(200,137)}, rotate = 90] [color={blue}  ,draw opacity=1 ][line width=3]    (20.77,-6.25) .. controls (13.2,-2.65) and (6.28,-0.57) .. (0,0) .. controls (6.28,0.57) and (13.2,2.65) .. (20.77,6.25)   ;
%Straight Lines [id:da6805495552081733] 
\draw [color={blue},draw opacity=1 ][line width=3]    (200,190) -- (155,235);
\draw [shift={(155,235)}, rotate = 315] [color={blue}  ,draw opacity=1 ][line width=3]    (20.77,-6.25) .. controls (13.2,-2.65) and (6.28,-0.57) .. (0,0) .. controls (6.28,0.57) and (13.2,2.65) .. (20.77,6.25)   ;
%Straight Lines [id:da6297654163979183] 
\draw [color={rgb, 255:red, 208; green, 2; blue, 27 }  ,draw opacity=1 ][fill={rgb, 255:red, 208; green, 2; blue, 27 }  ,fill opacity=1 ]   (160.19,44.83) -- (172.23,56.9) ;
%Straight Lines [id:da7095982443008217] 
\draw [color={rgb, 255:red, 208; green, 2; blue, 27 }  ,draw opacity=1 ][fill={rgb, 255:red, 208; green, 2; blue, 27 }  ,fill opacity=1 ]   (160.19,45.44) -- (171.6,32.4) ;
%Straight Lines [id:da4837099735594046] 
\draw [color={rgb, 255:red, 208; green, 2; blue, 27 }  ,draw opacity=1 ][fill={rgb, 255:red, 208; green, 2; blue, 27 }  ,fill opacity=1 ]   (160.19,44.83) -- (134.78,45.26) ;
%Straight Lines [id:da5379804498318606] 
\draw [color={rgb, 255:red, 208; green, 2; blue, 27 }  ,draw opacity=1 ][fill={rgb, 255:red, 208; green, 2; blue, 27 }  ,fill opacity=1 ]   (123.54,32.4) -- (134.78,45.26) ;
%Straight Lines [id:da1841990030626035] 
\draw [color={rgb, 255:red, 208; green, 2; blue, 27 }  ,draw opacity=1 ][fill={rgb, 255:red, 208; green, 2; blue, 27 }  ,fill opacity=1 ]   (134.78,45.26) -- (123.54,56.9) ;
%Straight Lines [id:da7516933766912692] 
\draw [color={rgb, 255:red, 208; green, 2; blue, 27 }  ,draw opacity=1 ][fill={rgb, 255:red, 208; green, 2; blue, 27 }  ,fill opacity=1 ]   (404.19,224.7) -- (416.23,236.78) ;
%Straight Lines [id:da545316209957059] 
\draw [color={rgb, 255:red, 208; green, 2; blue, 27 }  ,draw opacity=1 ]   (404.19,225.31) -- (415.6,212.28) ;
%Straight Lines [id:da5362069321342038] 
\draw [color={rgb, 255:red, 208; green, 2; blue, 27 }  ,draw opacity=1 ][fill={rgb, 255:red, 208; green, 2; blue, 27 }  ,fill opacity=1 ]   (404.19,224.7) -- (378.78,225.14) ;
%Straight Lines [id:da7268105413773782] 
\draw [color={rgb, 255:red, 208; green, 2; blue, 27 }  ,draw opacity=1 ][fill={rgb, 255:red, 208; green, 2; blue, 27 }  ,fill opacity=1 ]   (367.54,212.28) -- (378.78,225.14) ;
%Straight Lines [id:da4827428202871422] 
\draw [color={rgb, 255:red, 208; green, 2; blue, 27 }  ,draw opacity=1 ][fill={rgb, 255:red, 208; green, 2; blue, 27 }  ,fill opacity=1 ]   (378.78,225.14) -- (367.54,236.78) ;
%Straight Lines [id:da7325470760796895] 
\draw [color={rgb, 255:red, 208; green, 2; blue, 27 }  ,draw opacity=1 ][fill={rgb, 255:red, 208; green, 2; blue, 27 }  ,fill opacity=1 ]   (55.19,273.7) -- (67.23,285.78) ;
%Straight Lines [id:da1859719269481197] 
\draw [color={rgb, 255:red, 208; green, 2; blue, 27 }  ,draw opacity=1 ][fill={rgb, 255:red, 208; green, 2; blue, 27 }  ,fill opacity=1 ]   (55.19,274.31) -- (66.6,261.28) ;
%Straight Lines [id:da40942263457019323] 
\draw [color={rgb, 255:red, 208; green, 2; blue, 27 }  ,draw opacity=1 ][fill={rgb, 255:red, 208; green, 2; blue, 27 }  ,fill opacity=1 ]   (55.19,273.7) -- (29.78,274.14) ;
%Straight Lines [id:da14038567461150597] 
\draw [color={rgb, 255:red, 208; green, 2; blue, 27 }  ,draw opacity=1 ][fill={rgb, 255:red, 208; green, 2; blue, 27 }  ,fill opacity=1 ]   (18.54,261.28) -- (29.78,274.14) ;
%Straight Lines [id:da07335190194258745] 
\draw [color={rgb, 255:red, 208; green, 2; blue, 27 }  ,draw opacity=1 ][fill={rgb, 255:red, 208; green, 2; blue, 27 }  ,fill opacity=1 ]   (29.78,274.14) -- (18.54,285.78) ;

% Text Node
\draw (326,54.4) node [anchor=north west][inner sep=0.75pt][font=\Large,color={rgb, 255:red, 208; green, 2; blue, 27 }  ,opacity=1 ]    { $M_{0,4}^{\trop}\subseteq Q_{[4]}$};
% Text Node
\draw (250,206.4) node [anchor=north west][inner sep=0.75pt] [color = blue]   {$\mathbf{r}^{4}_{\{1,3\}} = \mathbf{v}_{\{1,3\}}$};
% Text Node
\draw (100,133.4) node [anchor=north west][inner sep=0.75pt][color = blue]    {$\mathbf{r}^{4}_{\{1,2\}} = \mathbf{v}_{\{1,2\}}$};
% Text Node
\draw (50,227.4) node [anchor=north west][inner sep=0.75pt] [color = blue]   {$\mathbf{r}^{4}_{\{2,3\}} = \mathbf{v}_{\{2,3\}}$};
% Text Node
\draw (112.99,22.53) node [anchor=north west][inner sep=0.75pt]  [font=\footnotesize,color={rgb, 255:red, 208; green, 2; blue, 27}  ,opacity=1 ]  {$1$};
% Text Node
\draw (112.37,54.99) node [anchor=north west][inner sep=0.75pt]  [font=\footnotesize,color={rgb, 255:red, 208; green, 2; blue, 27 }  ,opacity=1 ]  {$2$};
% Text Node
\draw (172.29,22.53) node [anchor=north west][inner sep=0.75pt]  [font=\footnotesize,color={rgb, 255:red, 208; green, 2; blue, 27 }  ,opacity=1 ]  {$3$};
% Text Node
\draw (172.92,53.15) node [anchor=north west][inner sep=0.75pt]  [font=\footnotesize,color={rgb, 255:red, 208; green, 2; blue, 27 }  ,opacity=1 ]  {$4$};
% Text Node
\draw (356.99,202.4) node [anchor=north west][inner sep=0.75pt]  [font=\footnotesize,color={rgb, 255:red, 208; green, 2; blue, 27 }  ,opacity=1 ]  {$1$};
% Text Node
\draw (418.37,199.86) node [anchor=north west][inner sep=0.75pt]  [font=\footnotesize,color={rgb, 255:red, 208; green, 2; blue, 27 }  ,opacity=1 ]  {$2$};
% Text Node
\draw (356.29,237.4) node [anchor=north west][inner sep=0.75pt]  [font=\footnotesize,color={rgb, 255:red, 208; green, 2; blue, 27 }  ,opacity=1 ]  {$3$};
% Text Node
\draw (416.92,233.03) node [anchor=north west][inner sep=0.75pt]  [font=\footnotesize,color={rgb, 255:red, 208; green, 2; blue, 27 }  ,opacity=1 ]  {$4$};
% Text Node
\draw (69.99,250.4) node [anchor=north west][inner sep=0.75pt]  [font=\footnotesize,color={rgb, 255:red, 208; green, 2; blue, 27 }  ,opacity=1 ]  {$1$};
% Text Node
\draw (7.37,283.86) node [anchor=north west][inner sep=0.75pt]  [font=\footnotesize,color={rgb, 255:red, 208; green, 2; blue, 27 }  ,opacity=1 ]  {$2$};
% Text Node
\draw (6.29,253.4) node [anchor=north west][inner sep=0.75pt]  [font=\footnotesize,color={rgb, 255:red, 208; green, 2; blue, 27 }  ,opacity=1 ]  {$3$};
% Text Node
\draw (67.92,282.03) node [anchor=north west][inner sep=0.75pt]  [font=\footnotesize,color={rgb, 255:red, 208; green, 2; blue, 27 }  ,opacity=1 ]  {$4$};

\end{tikzpicture}
    \caption{The space $M_{0,4}^{\trop}$ inside $Q_{[4]}$ as a balanced fan.}
    \label{fig:m04}
\end{figure}
The vector space $Q_{[4]}$ is two-dimensional, presented as generated by three vectors. Choosing $k = 4$, the generators are $\mathbf{r}^4_{\{1,2\}},$ $\mathbf{r}^4_{\{1,3\}},$ and $\mathbf{r}^4_{\{2,3\}}$, and the relation is 
\begin{equation}\label{eq:balm04}
 \mathbf{r}^4_{\{1,2\}}+ \mathbf{r}^4_{\{1,3\}}+ \mathbf{r}^4_{\{2,3\}} = 0.   
\end{equation}
One may define a linear isomorphism $Q_{[4]}\xrightarrow{\cong}{}\mathbb{R}^2$ by identifying the first two generators with the standard basis vectors; the third generator becomes the vector $(-1,-1)$.
The rays of $M_{0,4}^{\trop}$ are spanned by the primitive vectors $\mathbf{v}_{\{1,2\}},\mathbf{v}_{\{1,3\}},\mathbf{v}_{\{2,3\}}$, which in this simple case are equal to the corresponding vectors $\mathbf{r}^4_{\{1,2\}},\mathbf{r}^4_{\{1,3\}},\mathbf{r}^4_{\{2,3\}}$ generating $Q_{[4]}$.
Since $M_{0,4}^{\trop}$ is a one-dimensional fan, balancing needs to be checked only at the vertex (the unique $0$-dimensional face). Since all rays are given weight one, the balancing equation \eqref{eq:balancingcondition} is readily seen to reduce to \eqref{eq:balm04} in this case. 
\end{ex}

\subsubsection*{Tropicalization of  Chow Classes}
While tropical intersection theory is an actively developing field \cite{MikhalkinICM, AllermannRau2010, Katz, Shaw, Gross}  we make use of a limited portion of it, which we recall in the context of $\overline{M}_{0,n}$. We refer the reader to \cite{Katz} for proofs and for a more complete treatment.
 
 If $Y\subseteq\overline{M}_{0,n}$ is a pure $k$-dimensional cycle that meets each boundary stratum in the expected dimension, then its tropicalization is supported on the $k$-dimensional cones of $M_{0,n}^{\trop}$, with the coefficient of a given cone $\sigma\subseteq\M_{0,n}^{\trop}$ equal to the intersection number of $Y$ with the corresponding boundary stratum $\Delta_\sigma\subseteq\overline{M}_{0,n}$ \cite[Prop. 9.4]{Katz}. By the moving lemma, one may apply this construction directly to Chow classes, to obtain the following.
\begin{definition}\label{def:tropcycles}
    Let $\alpha\in A^k(\overline{M}_{0,n})$ be a codimension-$k$ Chow class. Define the tropicalization $ \alpha^{\trop}$ of $\alpha$ to be
    \begin{equation}
    \alpha^{\trop}  = \sum_{\sigma} \left(\int_{\Mbar_{0,n}} \Delta_\sigma \cdot \alpha\right)   \sigma,    
    \end{equation}
    where one may sum  over all cones of $M^{\trop}_{0,n}$, but the coefficients are nonzero only for codimension-$k$ cones.
\end{definition}
 \begin{fact}[\cite{Katz}, Lemma 8.13] The weighted fan $\alpha^{\trop}\subseteq Q_{[n]}$  from Definition \ref{def:tropcycles} is balanced. 
 \end{fact}
We refer to a balanced weighted subfan of $M_{0,n}^{\trop}$ as a {\it tropical cycle}.

\begin{ex}\label{ex:M05}
    We  consider the case $n=5$, and give an example of how to check the balancing condition along a ray. We then construct the tropicalization of the boundary divisor class $D_{\{1,2\}}\in A^1(\overline{M}_{0,5})$ and show it is a pure 1-dimensional weighted balanced fan.

    The vector space $Q_{[5]}$ is $5$-dimensional; choosing $k=5$, we have six generators $\mathbf{r}^{5}_{\{i,j\}}$ with $i,j\not=5$, which agree with six of the ten primitive vectors spanning the rays of $M_{0,5}^{\trop}$, i.e. \begin{equation}
        \label{eq:equalrays}
        \mathbf{v}_{\{i,j\}} = \mathbf{r}^{5}_{\{i,j\}}\ \ \ \ \mbox{for $i,j \not = 5$}.
    \end{equation} By \eqref{eq:relvr}, the primitive vectors for the remaining four rays of $M_{0,5}^{\trop}$ are
    \begin{align}\label{eq:expand}
        \mathbf{v}_{\{1,2,3\}} = \mathbf{r}^{5}_{\{1,2\}}+\mathbf{r}^{5}_{\{1,3\}}+\mathbf{r}^{5}_{\{2,3\}},\nonumber \\
         \mathbf{v}_{\{1,2,4\}} = \mathbf{r}^{5}_{\{1,2\}}+\mathbf{r}^{5}_{\{1,4\}}+\mathbf{r}^{5}_{\{2,4\}},\\
         \mathbf{v}_{\{1,3,4\}} = \mathbf{r}^{5}_{\{1,3\}}+\mathbf{r}^{5}_{\{1,4\}}+\mathbf{r}^{5}_{\{3,4\}},\nonumber\\
          \mathbf{v}_{\{2,3,4\}} = \mathbf{r}^{5}_{\{2,3\}}+\mathbf{r}^{5}_{\{2,4\}}+\mathbf{r}^{5}_{\{3,4\}}.\nonumber
    \end{align}

Each 2-dimensional cone of $M_{0,5}^{\trop}$ is spanned by two rays, and its image in $Q_{[5]}$ is determined by the images of the rays. We next check that giving all two-dimensional cones weight $1$, the fan $M_{0,5}^{\trop}$ is balanced along the ray $\tau$ spanned by $\mathbf{v}_{\{1,2,3\}}$, see Figure \ref{fig:balancingtau}. There are three two-dimensional cones containing $\tau$, denoted $\sigma_{\{1,2\}\{4,5\}}, \sigma_{\{1,3\}\{4,5\}}, \sigma_{\{2,3\}\{4,5\}}$. One may check that for each $\sigma_{\{i,j\}\{4,5\}}$, a normal vector to $\tau$ is $\mathbf{u}_{\tau/\sigma_{\{i,j\}\{4,5\}}} = \mathbf{v}_{\{i,j\}}$.
Then \eqref{eq:balancingcondition} becomes
\begin{equation}
     \mathbf{v}_{\{1,2\}}+\mathbf{v}_{\{1,3\}}+\mathbf{v}_{\{2,3\}} = 0 \in Q_{[5]}/\langle \mathbf{v}_{\{1,2,3\}}\rangle_{\mathbb{R}},
\end{equation}
which follows from \eqref{eq:equalrays},\eqref{eq:expand}. Thus $\M_{0,5}^{\trop}$ is balanced along $\mathbf{v}_{\{1,2,3\}}$.
\begin{figure}
    \centering

\tikzset{every picture/.style={line width=0.75pt}} %set default line width to 0.75pt        

\begin{tikzpicture}[x=0.75pt,y=0.75pt,yscale=-1,xscale=1]

%Straight Lines [id:da4496171677743388] 
\draw [color={rgb, 255:red, 208; green, 2; blue, 27 }  ,draw opacity=1 ][line width=2.25]    (199.43,15.43) -- (200.43,180.43) ;
%Straight Lines [id:da9642251814427869] 
\draw    (181.43,162.43) -- (305.43,284.43) ;
%Straight Lines [id:da25435345286977196] 
\draw [color={rgb, 255:red, 208; green, 2; blue, 27 }  ,draw opacity=1 ][line width=2.25]    (200.43,180.43) -- (399.43,183.43) ;
%Straight Lines [id:da11295041503421244] 
\draw [color={rgb, 255:red, 246; green, 148; blue, 148 }  ,draw opacity=1 ] [dash pattern={on 0.84pt off 2.51pt}]  (282.43,167.43) -- (198.43,49.43) ;
%Straight Lines [id:da7244803292593911] 
\draw [color={rgb, 255:red, 248; green, 161; blue, 161 }  ,draw opacity=1 ] [dash pattern={on 0.84pt off 2.51pt}]  (282.43,167.43) -- (283.43,263.43) ;
%Straight Lines [id:da3078265735062016] 
\draw [color={rgb, 255:red, 245; green, 164; blue, 164 }  ,draw opacity=1 ] [dash pattern={on 0.84pt off 2.51pt}]  (282.43,167.43) -- (349.43,181.43) ;
%Straight Lines [id:da0069726217226646625] 
\draw [color={rgb, 255:red, 208; green, 2; blue, 27 }  ,draw opacity=1 ][line width=2.25]    (399.43,150.43) -- (200.43,180.43) ;
%Straight Lines [id:da9927456836678381] 
\draw [color={rgb, 255:red, 208; green, 2; blue, 27 }  ,draw opacity=1 ][line width=2.25]    (200.43,180.43) -- (310.43,289.43) ;
%Straight Lines [id:da709405683717568] 
\draw    (170.43,179.43) -- (200.43,180.43) ;
%Straight Lines [id:da44773030071759035] 
\draw    (200.43,209.43) -- (200.43,180.43) ;
%Shape: Polygon [id:ds7132126908221081] 
\draw  [draw opacity=0][fill={rgb, 255:red, 208; green, 100; blue, 100 }  ,fill opacity=0.42 ][dash pattern={on 4.5pt off 4.5pt}] (198.43,49.43) -- (282.43,167.43) -- (200.43,180.43) -- (198.43,49.43) -- (198.43,49.43) -- cycle ;
%Shape: Polygon [id:ds273446240036993] 
\draw  [draw opacity=0][fill={rgb, 255:red, 247; green, 203; blue, 203 }  ,fill opacity=0.43 ] (200.43,180.43) -- (282.43,167.43) -- (283.43,263.43) -- (200.43,180.43) -- (200.43,180.43) -- cycle ;
%Shape: Polygon [id:ds8118867188137864] 
\draw  [draw opacity=0][fill={rgb, 255:red, 246; green, 105; blue, 105 }  ,fill opacity=0.46 ] (200.43,180.43) -- (282.43,167.43) -- (349.43,181.43) -- (200.43,180.43) -- (200.43,180.43) -- cycle ;
%Straight Lines [id:da820145862176783] 
\draw [color={rgb, 255:red, 208; green, 2; blue, 27 }  ,draw opacity=1 ]   (368,62.28) -- (344.31,72.51) ;
%Straight Lines [id:da05643075199675862] 
\draw [color={rgb, 255:red, 208; green, 2; blue, 27 }  ,draw opacity=1 ]   (385.56,73.33) -- (368,62.28) ;
%Straight Lines [id:da10031395267787735] 
\draw [color={rgb, 255:red, 208; green, 2; blue, 27 }  ,draw opacity=1 ]   (367.51,46.17) -- (368,62.28) ;
%Straight Lines [id:da7665743410089998] 
\draw [color={rgb, 255:red, 208; green, 2; blue, 27 }  ,draw opacity=1 ]   (385.56,73.33) -- (386.05,89.44) ;
%Straight Lines [id:da3066325387545541] 
\draw [color={rgb, 255:red, 208; green, 2; blue, 27 }  ,draw opacity=1 ]   (344.31,72.51) -- (344.8,88.61) ;
%Straight Lines [id:da5166005771449307] 
\draw [color={rgb, 255:red, 208; green, 2; blue, 27 }  ,draw opacity=1 ]   (403.61,73.33) -- (385.56,73.33) ;
%Straight Lines [id:da7230621597191047] 
\draw [color={rgb, 255:red, 208; green, 2; blue, 27 }  ,draw opacity=1 ]   (344.31,72.51) -- (326.26,72.51) ;
%Straight Lines [id:da043035407576865525] 
\draw [color={rgb, 255:red, 208; green, 2; blue, 27 }  ,draw opacity=1 ]   (525.45,139.24) -- (511.76,131.5) ;
%Straight Lines [id:da645039423882765] 
\draw [color={rgb, 255:red, 208; green, 2; blue, 27 }  ,draw opacity=1 ]   (511.37,120.24) -- (511.76,131.5) ;
%Straight Lines [id:da7522424975026214] 
\draw [color={rgb, 255:red, 208; green, 2; blue, 27 }  ,draw opacity=1 ]   (525.45,139.24) -- (525.84,150.5) ;
%Straight Lines [id:da5557444002311793] 
\draw [color={rgb, 255:red, 208; green, 2; blue, 27 }  ,draw opacity=1 ]   (511.99,131.66) -- (512.09,144.2) ;
%Straight Lines [id:da1305171771637148] 
\draw [color={rgb, 255:red, 208; green, 2; blue, 27 }  ,draw opacity=1 ]   (539.54,139.24) -- (525.45,139.24) ;
%Straight Lines [id:da47242520370557783] 
\draw [color={rgb, 255:red, 208; green, 2; blue, 27 }  ,draw opacity=1 ]   (511.21,131.66) -- (497.13,131.66) ;
%Straight Lines [id:da777696993580518] 
\draw [color={rgb, 255:red, 208; green, 2; blue, 27 }  ,draw opacity=1 ]   (131.86,29.06) -- (115.42,36.39) ;
%Straight Lines [id:da29566056528038565] 
\draw [color={rgb, 255:red, 208; green, 2; blue, 27 }  ,draw opacity=1 ]   (131.52,17.52) -- (131.86,29.06) ;
%Straight Lines [id:da6183291431843565] 
\draw [color={rgb, 255:red, 208; green, 2; blue, 27 }  ,draw opacity=1 ]   (131.55,29.09) -- (131.89,40.64) ;
%Straight Lines [id:da15252899516589147] 
\draw [color={rgb, 255:red, 208; green, 2; blue, 27 }  ,draw opacity=1 ]   (115.42,36.39) -- (115.76,47.93) ;
%Straight Lines [id:da9127971310731229] 
\draw [color={rgb, 255:red, 208; green, 2; blue, 27 }  ,draw opacity=1 ]   (144.08,29.09) -- (131.55,29.09) ;
%Straight Lines [id:da8165351299851697] 
\draw [color={rgb, 255:red, 208; green, 2; blue, 27 }  ,draw opacity=1 ]   (115.42,36.39) -- (102.9,36.39) ;

% Text Node
\draw (380,126.4) node [anchor=north west][inner sep=0.75pt]  [font=\normalsize]  {$\tau = \langle \mathbf{v}_{\{1,2,3\}}\rangle_{\mathbb{R}_{\geq 0}}$};
% Text Node
\draw (323,218.4) node [anchor=north west][inner sep=0.75pt]  [font=\footnotesize]  {$\sigma_{\{1,2\}\{4,5\}}$};
\draw[->] (323,218.4)--(275,210);
% Text Node
\draw (256,64.4) node [anchor=north west][inner sep=0.75pt]  [font=\footnotesize]  {$\sigma_{\{2,3\}\{4,5\}}$};
\draw[->] (266,80) -- (240,130);
% Text Node
\draw (354,163.4) node [anchor=north west][inner sep=0.75pt]  [font=\footnotesize]  {$\sigma_{\{1,3\}\{4,5\}}$};
\draw[->] (350,168)--(305,175);
% Text Node
\draw (319,279.4) node [anchor=north west][inner sep=0.75pt]    {$\langle\mathbf{v}_{\{1,2\}}\rangle_{\mathbb{R}_{\geq 0}}$};
% Text Node
\draw (410,173.4) node [anchor=north west][inner sep=0.75pt]    {$\langle\mathbf{v}_{\{1,3\}}\rangle_{\mathbb{R}_{\geq 0}}$};
% Text Node
\draw (140,-5) node [anchor=north west][inner sep=0.75pt]    {$\langle\mathbf{v}_{\{2,3\}}\rangle_{\mathbb{R}_{\geq 0}}$};
% Text Node
\draw (386.07,90.6) node [anchor=north west][inner sep=0.75pt]  [font=\scriptsize]  {$\textcolor[rgb]{0.82,0.01,0.11}{4}$};
% Text Node
\draw (408.41,71.67) node [anchor=north west][inner sep=0.75pt]  [font=\scriptsize]  {$\textcolor[rgb]{0.82,0.01,0.11}{5}$};
% Text Node
\draw (362,32.16) node [anchor=north west][inner sep=0.75pt]  [font=\scriptsize]  {$\textcolor[rgb]{0.82,0.01,0.11}{1}$};
% Text Node
\draw (313.02,69.2) node [anchor=north west][inner sep=0.75pt]  [font=\scriptsize]  {$\textcolor[rgb]{0.82,0.01,0.11}{2}$};
% Text Node
\draw (334.5,93.07) node [anchor=north west][inner sep=0.75pt]  [font=\scriptsize]  {$\textcolor[rgb]{0.82,0.01,0.11}{3}$};
% Text Node
\draw (524.53,149.93) node [anchor=north west][inner sep=0.75pt]  [font=\scriptsize]  {$\textcolor[rgb]{0.82,0.01,0.11}{4}$};
% Text Node
\draw (541.97,136.69) node [anchor=north west][inner sep=0.75pt]  [font=\scriptsize]  {$\textcolor[rgb]{0.82,0.01,0.11}{5}$};
% Text Node
\draw (505.76,109.06) node [anchor=north west][inner sep=0.75pt]  [font=\scriptsize]  {$\textcolor[rgb]{0.82,0.01,0.11}{1}$};
% Text Node
\draw (485.48,127.97) node [anchor=north west][inner sep=0.75pt]  [font=\scriptsize]  {$\textcolor[rgb]{0.82,0.01,0.11}{2}$};
% Text Node
\draw (502.24,144.66) node [anchor=north west][inner sep=0.75pt]  [font=\scriptsize]  {$\textcolor[rgb]{0.82,0.01,0.11}{3}$};
% Text Node
\draw (130.07,40.16) node [anchor=north west][inner sep=0.75pt]  [font=\scriptsize]  {$\textcolor[rgb]{0.82,0.01,0.11}{4}$};
% Text Node
\draw (145.57,26.6) node [anchor=north west][inner sep=0.75pt]  [font=\scriptsize]  {$\textcolor[rgb]{0.82,0.01,0.11}{5}$};
% Text Node
\draw (125.86,6.18) node [anchor=north west][inner sep=0.75pt]  [font=\scriptsize]  {$\textcolor[rgb]{0.82,0.01,0.11}{1}$};
% Text Node
\draw (91.87,32.71) node [anchor=north west][inner sep=0.75pt]  [font=\scriptsize]  {$\textcolor[rgb]{0.82,0.01,0.11}{2}$};
% Text Node
\draw (106.78,49.81) node [anchor=north west][inner sep=0.75pt]  [font=\scriptsize]  {$\textcolor[rgb]{0.82,0.01,0.11}{3}$};

\end{tikzpicture}
    \caption{A local picture of $M_{0,5}^{\trop}$ around the ray $\tau = \langle \mathbf{v}_{\{1,2,3\}}\rangle_{\mathbb{R}_{\geq 0}}$. The red trees drawn next to the rays $\tau, \langle \mathbf{v}_{\{2,3\}}\rangle_{\mathbb{R}_{\geq 0}}$ and the two dimensional cone $\sigma_{\{2,3\}\{4,5\}}$ show the tropical curves parameterized by those cones.}
    \label{fig:balancingtau}
\end{figure}
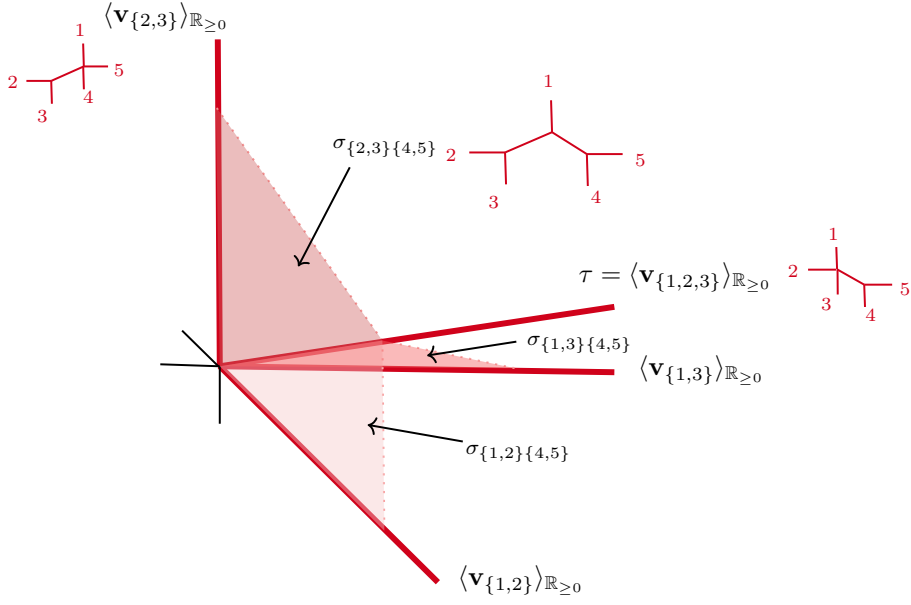

Consider  the class $D_{\{1,2\}} \in A^1(\overline{M}_{0,5})$ of the divisor generically parameterizing stable curves with two components, one containing the first two marks, the other containing the other marks. We construct its tropicalization $D_{\{1,2\}}^{\trop}$ and check it is balanced.

The intersections of boundary strata in $\overline{M}_{0,n}$ are well understood (see for example \cite{RenzonotesonM0n}). The intersection numbers of $D_{\{1,2\}}$ with all boundary divisors are as follows: \begin{align*}
    &\deg(D_{\{1,2\}}^2) = -1,\\
    \deg(D_{\{1,2\}}\cdot D_{\{3,4\}})  = &\deg(D_{\{1,2\}}\cdot D_{\{3,5\}}) = \deg(D_{\{1,2\}}\cdot D_{\{4,5\}})= 1,
\end{align*}
and the intersection numbers of $D_{\{1,2\}}$ with the remaining $6$ boundary divisors are all zero.

By Definition \ref{def:tropcycles}, we have
\begin{equation}\label{eq:CoeffsOfD12}
  D_{\{1,2\}}^{\trop} =  
  \langle \mathbf{v}_{\{3,4\}}\rangle_{\mathbb{R}_{\geq 0}}+
  \langle \mathbf{v}_{\{1,2,3\}}\rangle_{\mathbb{R}_{\geq 0}}+
  \langle \mathbf{v}_{\{1,2,4\}}\rangle_{\mathbb{R}_{\geq 0}}-
  \langle \mathbf{v}_{\{1,2\}}\rangle_{\mathbb{R}_{\geq 0}}.
\end{equation}
To check that $D_{\{1,2\}}^{\trop}$ satisfies the balancing condition at the origin, we write each of the primitive vectors of the rays of $D_{\{1,2\}}^{\trop}$ as linear combinations of the generators $\mathbf{r}^5_{\{i,j\}}$ using \eqref{eq:relvr}. The computation is done in Table \ref{tab:balforD12trop}.  Indeed, Table \ref{tab:balforD12trop} shows that the sum of the weighted primitive normal vectors of the rays of $D_{\{1,2\}}^{\trop}$ can be expressed as $\sum_{\{i,j\}\in[5]\setminus\{5\}}\mathbf r^5_{\{i,j\}},$ which is equal to zero by \eqref{eq:uniquerelations}. See also Example \ref{ex:M05Addendum} below.
\begin{table}[bt]
    \centering
    \begin{tabular}{|c|c|c|c|c|c|c|}
    \hline
         & $\mathbf{r}^5_{\{1,2\}} $ 
         & $\mathbf{r}^5_{\{1,3\}} $
         & $\mathbf{r}^5_{\{1,4\}} $
         & $\mathbf{r}^5_{\{2,3\}} $
         & $\mathbf{r}^5_{\{2,4\}} $
         & $\mathbf{r}^5_{\{3,4\}} $\\
         \hline \hline
       $\mathbf{v}_{\{3,4\}}$ &0 &0 &0 &0 &0 &1 \\ \hline
        $\mathbf{v}_{\{1,2,3\}}$ &1 &1 &0 &1 &0 &0 \\ \hline
         $\mathbf{v}_{\{1,2,4\}}$ &1 &0 &1 &0 &1 &0 \\ \hline
          $-\mathbf{v}_{\{1,2\}}$ & -1 &0 &0 &0 &0 &0 \\ \hline \hline
 sum & 1&1 &1 &1 &1 &1 \\\hline
          \end{tabular}
    \caption{The computation that $D_{\{1,2\}}^{\trop}$ is a balanced cycle. The first four rows of the tables write the coefficients of the linear combination expressing each of the weighted primitive vectors in terms  of the generators $\mathbf{r}^5_{\{i,j\}}$. The last row, which adds up the previous ones,  computes equation \eqref{eq:balancingcondition} at the origin: one obtains the relation among the $\mathbf{r}^5_{\{i,j\}}$'s. }
    \label{tab:balforD12trop}
\end{table}

\end{ex}

We now show how to use the balancing condition to obtain explicit linear  relations among the coefficients of a tropical cycle on $\M_{0,n}^{\trop}$. We describe the process for a one-dimensional tropical cycle, which is the case that will be used the most later on.

Consider a 1-dimensional tropical cycle  $A = \sum_{I} a_I \langle \mathbf{v}_I \rangle_{\mathbb{R}_\geq 0}$ on $\M_{0,n}^{\trop}$. We are using the notation introduced in this section; that is, we have fixed $k\in[n]$, and we write $\mathbf{v}_I=\sum_{\{i,j\subseteq I\}}\mathbf{r}_{\{i,j\}}^k$ for the image in $Q_{[n]}$ of the tropical curve with one edge of length one, separating the marks into sets $I,I^c$ with $k\in I^c$. The origin is the only zero-dimensional cone of $A$, so checking balancing for this cycle amounts to verifying that 
\begin{equation}\label{eq:lhsofbal}
  \sum_{\substack{I\subseteq [n]\smallsetminus\{k\}\\2\le\abs{I}\le n-2}} a_I \mathbf{v}_I   = 0 \in Q_{[n]}.
\end{equation}
We use the same mark $k$ to pick a set of generators for $Q_{[n]}$, and  rewrite each of the $\mathbf{v}_I$ using equation \eqref{eq:relvr}, so  that the left hand side of \eqref{eq:lhsofbal} becomes:
\begin{equation}\label{eq:quasiboom}
   \sum_{\substack{I\subseteq [n]\setminus\{k\}\\2\le\abs{I}\le n-2}} a_I \left(\sum_{\{i,j\}\subseteq I} \mathbf{r}_{\{i,j\}}^k \right)
    =
\sum_{\{i,j\}\in[n]\setminus\{k\}}
\left(
\sum_{\substack{\{i,j\}\subseteq I\subseteq [n]\setminus\{k\}\\ \abs{I}\le n-2}} a_I 
\right)
\mathbf{r}^k_{\{i,j\}}
    =
    \sum_{\{i,j\}\in[n]\setminus\{k\}} B^k_{\{i,j\}} \mathbf{r}^k_{\{i,j\}},
\end{equation}
where the notation $B^k_{\{i,j\}}$ is defined by the second equality. Equation \eqref{eq:quasiboom}
 equals the zero vector if and only if for every pair $\{i_1, j_1\}, \{i_2, j_2\}$ with $i_1,j_1,i_2,j_2\ne k$ we have 
\begin{equation}\label{eq:boomerang}
    B^k_{\{i_1, j_1\}} = B^k_{\{i_2, j_2\}}.
\end{equation}
We observe that \eqref{eq:boomerang} is a linear homogeneous equation in the coefficients of the tropical cycle $A$. We call any such equation a {\it balancing relation}.

There is a simple description for the coefficient  $B^k_{\{i, j\}}$: it is obtained by adding the coefficient $a_{\{i, j\}}$ of the ray $\langle \mathbf{v}_{\{i, j\}} \rangle_{\mathbb{R}_\geq 0}$ and  the coefficients of the $2^{n-3}-2$ rays $\langle \mathbf{v}_I\rangle_{\mathbb{R}_\geq 0}$ adjacent to it\footnote{Here $\mathbf v_I$ and $\mathbf v_{\{i,j\}}$ being adjacent means that they span a two dimensional cone in $\M_{0,n}^{\trop}$.} such that $i,j\in I$. 
\begin{ex}\label{ex:M05Addendum}
    We illustrate how to use \eqref{eq:boomerang} concretely. As in Example \ref{ex:M05}, let $n=5$ and choose $k=5.$ By \eqref{eq:quasiboom} we have
    \begin{align*}
        B_{\{1,2\}}^5&=a_{\{1,2\}}+a_{\{1,2,3\}}+a_{\{1,2,4\}}
        \\
        B_{\{3,4\}}^5&=a_{\{3,4\}}+a_{\{1,3,4\}}+a_{\{2,3,4\}},
    \end{align*}
    so \eqref{eq:boomerang} in this case yields $$a_{\{1,2\}}+a_{\{1,2,3\}}+a_{\{1,2,4\}}=a_{\{3,4\}}+a_{\{1,3,4\}}+a_{\{2,3,4\}}.$$
    We observe that this is consistent with our calculations in Example \ref{ex:M05}; for $D_{\{1,2\}}^{\trop}$ the above equation reads: $$-1+1+1=1+0+0.$$
    We illustrate $B_{\{1,2\}}^5$ and $B_{\{3,4\}}^5$ in Figure \ref{fig:boomerangs}, and note it amounts to computing the sums of the first and last columns of Table \ref{tab:balforD12trop}.
\end{ex}

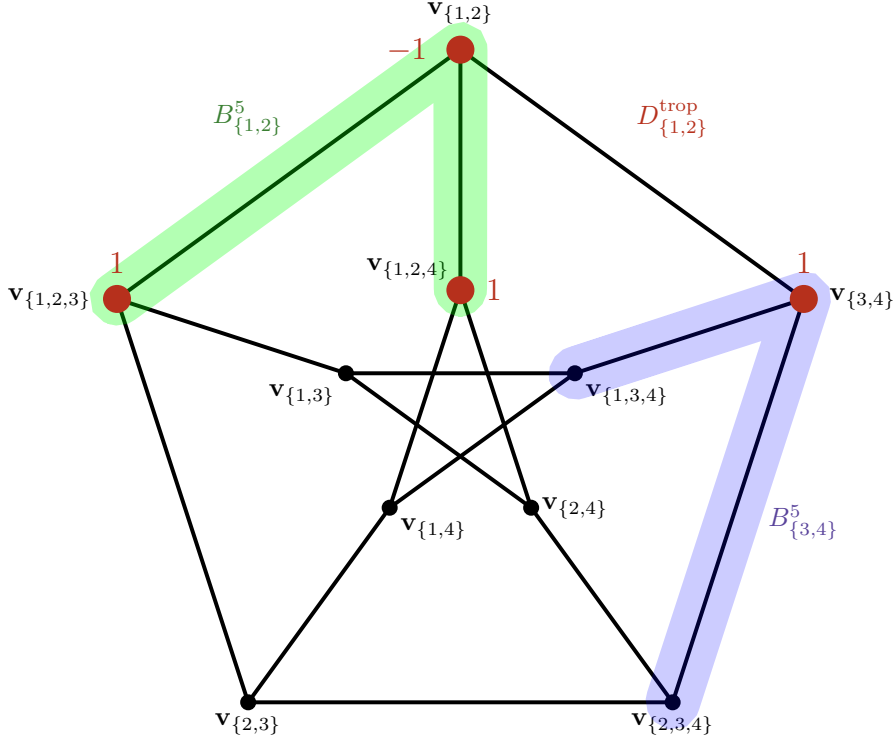
\begin{figure}
    \centering

\tikzset{every picture/.style={line width=1.5pt}} %set default line width to 0.75pt        

\begin{tikzpicture}[scale=1.6]
    \foreach \i in {0,1,2,3,4} {
    \draw (18+72*\i:1)--(162+72*\i:1);
    \draw (18+72*\i:1)--(18+72*\i:3);
    \draw (18+72*\i:3)--(90+72*\i:3);
    \filldraw (18+72*\i:1) circle(.05);
    \filldraw (18+72*\i:3) circle(.05);
    };
    \draw (18+72*0:1) node[below right] {$\mathbf{v}_{\{1,3,4\}}$};
    \draw (18+72*1:1) node[above left] {$\mathbf{v}_{\{1,2,4\}}$};
    
    \draw (18+72*2:1) node[below left] {$\mathbf v_{\{1,3\}}$};
    \draw (18+72*3:1) node[below right] {$\mathbf v_{\{1,4\}}$};
    \draw (18+72*4:1) node[right] {$\mathbf v_{\{2,4\}}$};
    \draw (18+72*0:3) node[right=.2cm] {$\mathbf v_{\{3,4\}}$};
    
    \draw (18+72*1:3) node[above=.2cm] {$\mathbf v_{\{1,2\}}$};
    
    \draw (18+72*2:3) node[left=.2cm] {$\mathbf{v}_{\{1,2,3\}}$};
    
    \draw (18+72*3:3) node[below] {$\mathbf v_{\{2,3\}}$};
    \draw (18+72*4:3) node[below] {$\mathbf{v}_{\{2,3,4\}}$};
    \draw[blue,opacity=0.2,line width=20,rounded corners=1] (18+72*5:2.99)--(18+72*4:2.99)--(18+72*4:3.01)--(18.1+72*5:3.01)--(18.1+72*5:1)--(17.9+72*5:1)--(17.9+72*5:3);
    \draw[green,opacity=0.3,line width=20,rounded corners=1] (18+72*1:2.99)--(18+72*2:2.99)--(18+72*2:3.01)--(17.9+72*1:3.01)--(17.9+72*1:1)--(18.1+72*1:1)--(18.1+72*1:3);
    \draw[OliveGreen] (18+72*1.5:3) node {$B_{\{1,2\}}^5$};
    \draw[Violet] (18+72*4.5:3) node {$B_{\{3,4\}}^5$};
    \filldraw[BrickRed] (18+72*0:3) circle(.1);
    \filldraw[BrickRed] (18+72*1:3) circle(.1);
    \filldraw[BrickRed] (18+72*2:3) circle(.1);
    \filldraw[BrickRed] (18+72*1:1) circle(.1);
    \draw[BrickRed] (18+72*1:1) node[right=.2cm] {\Large $1$};
    \draw[BrickRed] (18+72*0:3) node[above=.2cm] {\Large $1$};
    \draw[BrickRed] (18+72*1:3) node[left=.3cm] {\Large $-1$};
    \draw[BrickRed] (18+72*2:3) node[above=.2cm] {\Large $1$};
    \draw[BrickRed] (18+72*.5:3) node {$D_{\{1,2\}}^{\trop}$};
\end{tikzpicture}
\caption{The Petersen graph gives a slice of the cone complex of $M_{0,5}^{\trop}$, parameterizing  tropical curves where the total length of all the edges is equal to one (unfortunately it is not possible to draw a two-dimensional projection of $M_{0,5}^{\trop}$ which allows simultaneously to observe the  combinatorial structure of the cone complex as well as the structure of balanced fan). The vertices correspond to the primitive vectors spanning the rays of $M_{0,5}^{\trop}$ and are labeled by the corresponding $\mathbf{v}_I$'s. In red we have the cycle $D_{\{1,2\}}^{\trop}$, where the coefficients written next to the dots are the weights of the corresponding rays. Rays without a red dot have weight $0$. The sum of the coefficients in the shaded areas in green (resp. in blue) gives the coefficient $B^5_{\{1,2\}}$ (resp. $B^5_{\{3,4\}}$). One can observe the balancing relation $B^5_{\{1,2\}}= B^5_{\{3,4\}}=1$ in this case.}
\label{fig:boomerangs}
\end{figure}

\subsection{Balancing and reconstruction}\label{sec:Balancing}

In this section we apply the techniques of Section \ref{sec:tctc} 
to the genus-zero part of any cohomological field theory (CohFT). At a basic level, a CohFT should be thought of as a way to obtain a collection of Chow classes $\Omega_{g,n}(v_1, \ldots, v_n)\in A^\ast(\overline{M}_{g,n})$, indexed by elements $v_i$ of some vector space, that behave recursively when restricted to any boundary stratum. We  refer the reader to the introductory paper \cite{PandhaCohCa} for the precise definitions. Let $\Omega$ be an arbitrary CohFT over a vector space $V$.  In order to simplify the exposition and notation, we impose the following assumption.

\begin{assumption} \label{ass:pd}
There exists a basis $\{e_\alpha\}_{\alpha \in A}$ of $V$ such that any class of the form $\Omega_{g,n}(e_{\alpha_1}, \ldots, e_{\alpha_n})$ is of pure dimension, denoted $d_{g,n, \alpha_1, \ldots, \alpha_n}$.
\end{assumption}
Assumption \ref{ass:pd} is not conceptually necessary; one may decompose any mixed-degree class $\Omega_{g,n}(v_1, \ldots, v_n)$ into its homogeneous parts, apply the constructions that follow to each homogeneous part, and then formally add everything up. Assumption \ref{ass:pd} essentially avoids having to carry around these sums.
Many CohFTs that are constructed from geometric properties of curves satisfy Assumption \ref{ass:pd}, including the main example in this work, the CohFT of Witten's $r$-spin classes (see Section \ref{sec:RSpin}). 
%%%%%%%
%% DEF: NUMERICAL PART OF COHFT
%%%%%%%%%
\begin{definition}
    The {\bf numerical part}  $\omega$ of a CohFT $\Omega$ records the degree of the zero-dimensional classes of $\Omega$, i.e.  
    \begin{equation}
        \omega_{g,n}(v_1, \ldots, v_n) = \int_{\Mbar_{g,n}} \Omega_{g,n}(v_1, \ldots, v_n).
    \end{equation}
  We call elements of the numerical part of a CohFT {\bf numerical CohFT invariants}. 
\end{definition}

Given  a CohFT class $\Omega_{g,n}(v_1, \ldots, v_n)$ of pure dimension $d_{g,n,v_1, \ldots, v_n }$, one may construct its  tropicalization following Definition \ref{def:tropcycles}:

\begin{equation} \label{eq:trop}
 \Omega^{\trop}_{g,n}(v_1, \ldots, v_n)  = \sum_{\sigma} \left(\int_{\Mbar_{g,n}} \Delta_\sigma \cdot \Omega_{g,n}(v_1, \ldots, v_n)\right)\   \sigma,
\end{equation}
where the sum runs over all cones $\sigma$ of $M_{g,n}^{\trop}$.

\begin{prop}\label{thm:TropicalCycle}
The class $\Omega^{\trop}_{g,n}(v_1, \ldots, v_n)$ is supported on the $(d_{g,n,v_1, \ldots, v_n })$-dimensional skeleton of $M_{g,n}^{\trop}$, and the nonzero coefficients are determined by the numerical part of $\Omega$. 
\end{prop}

\begin{proof}
    The first part of the statement is simply a dimension count: $\Omega^{\trop}_{g,n}(v_1, \ldots, v_n)$ can only intersect nontrivially strata of codimension $d_{g,n,v_1, \ldots, v_n }$, which in turn correspond to cones of $M_{g,n}^{\trop}$ of dimension $d_{g,n,v_1, \ldots, v_n }$.
    The second part follows from the splitting axioms for CohFTs (see \cite[Sec. 0.3]{PandhaCohCa}), which imply that the restriction of the class $\Omega^{\trop}_{g,n}(v_1, \ldots, v_n)$ to a boundary stratum $\Delta_\sigma = (\gl_\sigma)_\ast ([\prod_i \overline{M}_{g_i, n_i}])$ is equal to the pushforward of a linear combination of CohFT classes pulled-back from the factors of $\prod_i \overline{M}_{g_i, n_i}$. By Fubini's theorem, the integral over $\overline{M}_{g,n}$ of each summand splits as a product of integrals over the individual factors. Thus, for each factor, we obtain a corresponding numerical CohFT invariant.
\end{proof}

Moduli spaces of tropical curves of positive genus cannot be given the structure of a balanced fan in a vector space. Therefore for the remainder of the paper, we restrict to $g=0$, where we can exploit the additional structure of balancing of tropical cycles. 

\begin{ex}\label{ex:M05RSpin}
We illustrate a balancing relation for the $r$-spin CohFT with $r=10$, using the notation of Section \ref{sec:RSpin} and Examples \ref{ex:M05} and \ref{ex:M05Addendum}. Consider the Witten class $W_{10}(3,4,5,5,6).$ By \eqref{eq:RankOfWittenBundle}, $W_{10}(3,4,5,5,6)$ is a cycle on $\Mbar_{0,5}$ of codimension 1 (and dimension 1).

We use the balancing condition for $W_{10}(3,4,5,5,6)^{\trop}$ to deduce a relationship between 3- and 4-point $10$-spin invariants. As in Example \ref{ex:M05}, we choose $k=5$ to obtain six generators $\mathbf r_{\{i,j\}}^5$ with $i,j\in [4] = [5]\setminus\{5\}.$ We have by Definition \ref{def:tropcycles} $$W_{10}(3,4,5,5,6)^{\trop}=\sum_{\substack{
I\subseteq[4] \\ 2\le|I|\le 3}}a_I\mathbf v_I,$$ where $a_I=\int_{\Mbar_{0,5}}D_{I}\cdot W_{10}(3,4,5,5,6).$ 
By \eqref{eq:quasiboom} and \eqref{eq:boomerang}, with $(i_1,j_1)=(1,2)$ and $(i_2,j_2)=(3,4)$, we have $B_{\{1,2\}}^5=B_{\{3,4\}}^5,$ where: \begin{align}\label{eq:exB125}
    B^5_{\{1,2\}}=\sum_{\substack{I\subseteq[4]\\1,2\in I,\thickspace\abs{I}\le3}}a_I=a_{\{1,2\}}+a_{\{1,2,3\}}+a_{\{1,2,4\}}
\end{align} and \begin{align}\label{eq:exB345}
    B^5_{\{3,4\}}=\sum_{\substack{I\subseteq[4]\\3,4\in I,\thickspace\abs{I}\le3}}a_I=a_{\{3,4\}}+a_{\{1,3,4\}}+a_{\{2,3,4\}}.
\end{align} By the CohFT Axiom (Proposition \ref{ax:RestrictToStratum}), we have e.g. \begin{align*}
    a_{\{2,3,4\}}=\int_{\Mbar_{0,5}}W_{10}(3,4,5,5,6)|_{D_{\{2,3,4\}}}&=\int_{\Mbar_{0,5}}(W_{10}(4,5,5,8)\boxtimes W_{10}(3,6,2))\\
    &=\left(\int_{\Mbar_{0,4}}W_{10}(4,5,5,8)\right)\left(\int_{\Mbar_{0,3}}W_{10}(3,6,2)\right)\\
    &=w_{10}(4,5,5,8)w_{10}(3,6,2).
\end{align*}
Here in the notation of Proposition \ref{ax:RestrictToStratum} we have $m_J=8$ and $m_{J^c}=2.$ Similarly we compute the other terms of \eqref{eq:exB125} and \eqref{eq:exB345}:
\begin{align*}
    a_{\{1,2\}}&=w_{10}(3,4,4)w_{10}(5,5,6,6)&a_{\{1,2,3\}}&=w_{10}(3,4,5,10)w_{10}(5,6,10)\\
    a_{\{1,2,4\}}&=w_{10}(3,4,5,10)w_{10}(5,6,10)&
    a_{\{3,4\}}&=w_{10}(5,5,1)w_{10}(3,4,6,9)\\a_{\{1,3,4\}}&=w_{10}(3,5,5,9)w_{10}(4,6,1).
\end{align*}
Thus the equation $B_{\{1,2\}}^5=B_{\{3,4\}}^5$ gives us a polynomial relation between genus-zero 10-spin invariants:
\begin{align}\label{eq:exRelation}
w_{10}(3,4,4)&w_{10}(5,5,6,6)+w_{10}(3,4,5,10)w_{10}(5,6,10)+w_{10}(3,4,5,10)w_{10}(5,6,10)\\
&=w_{10}(5,5,1)w_{10}(3,4,6,9)+w_{10}(3,5,5,9)w_{10}(4,6,1)+w_{10}(4,5,5,8)w_{10}(3,6,2).\nonumber
\end{align}
We can generate a huge number of relations of this form between genus-zero 10-spin invariants, both by varying $i_1,$ $j_1$, $i_2,$ $j_2,$ and $k$, and by considering balancing of other (including higher-dimensional) 10-spin Witten classes along various cones of $\M_{0,n}^{\trop}$. Theorem \ref{thm:wdvv}  shows that this collection of relations is equivalent to the WDVV equations, and in Section \ref{sec:recursivestructure}, we use the balancing relations to give an efficient recursive algorithm for reconstructing any genus-zero $r$-spin invariant. For the moment,  we  note that \eqref{eq:exRelation} is consistent with Proposition \ref{prop:34Point}; indeed, \eqref{eq:exRelation} reads $$1\cdot4+0\cdot1+0\cdot1=1\cdot1+1\cdot1+2\cdot1.$$ In Figure \ref{fig:34556}, we have computed $a_I$ for all $I$, and labeled the rays of $\M_{0,5}^{\trop}$ accordingly. We have then drawn the quantities $B^k_{\{i,j\}}$ for all $i,j\in[5]\setminus\{k\}$, for $k=5$ (left picture) and $k=1$ (right picture). Pictorially, one can immediately confirm the balancing relations for $B^5_{\{i,j\}}$ and $B^1_{\{i,j\}}.$

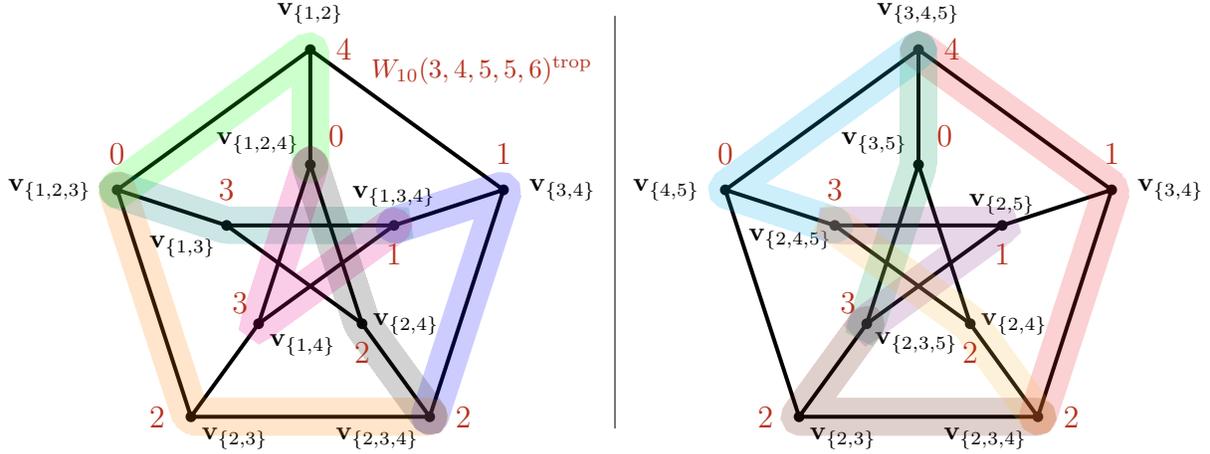
\begin{figure}[tb]
    \centering

\tikzset{every picture/.style={line width=1.5pt}} %set default line width to 0.75pt        

\begin{tikzpicture}[scale=.9]
    \foreach \i in {0,1,2,3,4} {
    \draw (18+72*\i:1.3)--(162+72*\i:1.3);
    \draw (18+72*\i:1.3)--(18+72*\i:3);
    \draw (18+72*\i:3)--(90+72*\i:3);
    \filldraw (18+72*\i:1.3) circle(.05);
    \filldraw (18+72*\i:3) circle(.05);
    };
    \draw (18+72*0:1.3) node[above=3pt] {$\mathbf{v}_{\{1,3,4\}}$};
    \draw (18+72*1:1.3) node[above left] {$\mathbf{v}_{\{1,2,4\}}$};
    
    \draw (18+72*2:1.3) node[below left] {$\mathbf v_{\{1,3\}}$};
    \draw (18+72*3:1.3) node[below right] {$\mathbf v_{\{1,4\}}$};
    \draw (18+72*4:1.3) node[right] {$\mathbf v_{\{2,4\}}$};
    \draw (18+72*0:3) node[right=.2cm] {$\mathbf v_{\{3,4\}}$};
    
    \draw (18+72*1:3) node[above=.2cm] {$\mathbf v_{\{1,2\}}$};
    
    \draw (18+72*2:3) node[left=.2cm] {$\mathbf{v}_{\{1,2,3\}}$};
    
    \draw (18+72*3:3) node[below right] {$\mathbf v_{\{2,3\}}$};
    \draw (18+72*4:3) node[below left] {$\mathbf{v}_{\{2,3,4\}}$};
    \draw[blue,opacity=0.2,line width=14,rounded corners=1] (18+72*5:3.01)--(18+72*4:3.01)--(18+72*4:2.99)--(18.01+72*5:2.99)--(18.01+72*5:1.3)--(17.99+72*5:1.3)--(17.99+72*5:3);
    %\draw[Violet] (18+72*4.5:3) node {$B_{\{3,4\}}^5$};
    \draw[green,opacity=0.2,line width=14,rounded corners=1] (18+72*1:3.01)--(18+72*2:3.01)--(18+72*2:2.99)--(18.01+72*1:2.99)--(18.01+72*1:1.3)--(17.99+72*1:1.3)--(17.99+72*1:3);
    %\draw[OliveGreen] (18+72*1.5:3) node {$B_{\{1,2\}}^5$};
    \draw[orange,opacity=0.2,line width=14,rounded corners=1] (18+72*3:3.01)--(18+72*2:3.01)--(18+72*2:2.99)--(18+72*3:2.99)--(18+72*4:2.99)--(18+72*4:3.01)--(18+72*3:3.01);
    %\draw[Mahogany] (18+72*2.5:3) node {$B_{\{1,2\}}^5$};
    \draw[teal,opacity=0.2,line width=14,rounded corners=1] (17.99+72*2:1.3)--(17.99+72*2:3)--(18.01+72*2:3)--(18.01+72*2:1.31)--(18+72*0:1.29)--(18+72*0:1.31)--(18+72*2:1.31);
    %\draw[Blue] (150:1.3.8) node {$B_{\{1,2\}}^5$};
    \draw[Magenta,opacity=0.2,line width=14,rounded corners=1] (18.01+72*3:1.3)--(18.01+72*1:1.3)--(17.99+72*1:1.3)--(17.99+72*3:1.3)--(17.99+72*0:1.3)--(18.01+72*0:1.3)--(18.01+72*3:1.3);
    %\draw[Bittersweet] (215:1.3.8) node {$B_{\{1,2\}}^5$};
    \draw[Black,opacity=0.2,line width=14,rounded corners=1] (18.01+72*4:1.3)--(18.01+72*1:1.3)--(17.99+72*1:1.3)--(17.99+72*4:1.3)--(17.99+72*4:3)--(18.01+72*4:3)--(18.01+72*4:1.3);
    \draw[BrickRed] (18+72*0:3) node[above=.2cm] {\Large $1$};
    \draw[BrickRed] (18+72*1:3) node[right=.2cm] {\Large $4$};
    \draw[BrickRed] (18+72*2:3) node[above=.2cm] {\Large $0$};
    \draw[BrickRed] (18+72*3:3) node[left=.2cm] {\Large $2$};
    \draw[BrickRed] (18+72*4:3) node[right=.2cm] {\Large $2$};
    \draw[BrickRed] (18+72*0:1.3) node[below=.1cm] {\Large $1$};
    \draw[BrickRed] (18+72*1:1.3) node[above right=.1cm] {\Large $0$};
    \draw[BrickRed] (18+72*2:1.3) node[above=.2cm] {\Large $3$};
    \draw[BrickRed] (18+72*3:1.3) node[above left] {\Large $3$};
    \draw[BrickRed] (18+72*4:1.3) node[below=.1cm] {\Large $2$};
    \draw[BrickRed] (18+72*.4:3.7) node {$W_{10}(3,4,5,5,6)^{\trop}$};
    \draw[thin] (4.5,-2.6)--(4.5,3.5);
\end{tikzpicture}
\begin{tikzpicture}[scale=.9]
    \foreach \i in {0,1,2,3,4} {
    \draw (18+72*\i:1.3)--(162+72*\i:1.3);
    \draw (18+72*\i:1.3)--(18+72*\i:3);
    \draw (18+72*\i:3)--(90+72*\i:3);
    \filldraw (18+72*\i:1.3) circle(.05);
    \filldraw (18+72*\i:3) circle(.05);
    };
    \draw (18+72*0:1.3) node[above] {$\mathbf{v}_{\{2,5\}}$};
    \draw (18+72*1:1.3) node[above left] {$\mathbf{v}_{\{3,5\}}$};
    
    \draw (18+72*2:1.3) node[below left=-3pt] {$\mathbf v_{\{2,4,5\}}$};
    \draw (18+72*3:1.3) node[below right=-1pt] {$\mathbf v_{\{2,3,5\}}$};
    \draw (18+72*4:1.3) node[right] {$\mathbf v_{\{2,4\}}$};
    \draw (18+72*0:3) node[right=.2cm] {$\mathbf v_{\{3,4\}}$};
    
    \draw (18+72*1:3) node[above=.2cm] {$\mathbf v_{\{3,4,5\}}$};
    
    \draw (18+72*2:3) node[left=.2cm] {$\mathbf{v}_{\{4,5\}}$};
    
    \draw (18+72*3:3) node[below right] {$\mathbf v_{\{2,3\}}$};
    \draw (18+72*4:3) node[below left] {$\mathbf{v}_{\{2,3,4\}}$};
    \draw[Red,opacity=0.2,line width=14,rounded corners=1] (18+72*5:3.01)--(18+72*4:3.01)--(18+72*4:2.99)--(18+72*5:2.99)--(18+72*1:2.99)--(18+72*1:3.01)--(18+72*5:3.01);
    %\draw[Violet] (18+72*4.5:3) node {$B_{\{3,4\}}^5$};
    \draw[ProcessBlue,opacity=0.2,line width=14,rounded corners=1] (18+72*2:3.01)--(18+72*1:3.01)--(18+72*1:2.99)--(18.01+72*2:2.99)--(18.01+72*2:1.3)--(17.99+72*2:1.3)--(17.99+72*2:3);
    %\draw[OliveGreen] (18+72*1.5:3) node {$B_{\{1,2\}}^5$};
    \draw[Sepia,opacity=0.2,line width=14,rounded corners=1] (18.01+72*3:3)--(18.01+72*3:1.3)--(17.99+72*3:1.3)--(17.99+72*3:3.01)--(18+72*4:3.01)--(18+72*4:2.99)--(18+72*3:2.99);
    %\draw[Mahogany] (18+72*2.5:3) node {$B_{\{1,2\}}^5$};
    \draw[ForestGreen,opacity=0.2,line width=14,rounded corners=1] (18.01+72*1:1.3)--(18.01+72*1:3)--(17.99+72*1:3)--(17.99+72*1:1.31)--(18+72*3:1.31)--(18+72*3:1.29)--(18+72*1:1.29);
    %\draw[Blue] (150:1.3.8) node {$B_{\{1,2\}}^5$};
    \draw[Fuchsia,opacity=0.2,line width=14,rounded corners=1] (18+72*0:1.31)--(18+72*2:1.31)--(18+72*2:1.29)--(18+72*0:1.29)--(18+72*3:1.29)--(18+72*3:1.31)--(18+72*0:1.31);
    %\draw[Bittersweet] (215:1.3.8) node {$B_{\{1,2\}}^5$};
    \draw[Dandelion,opacity=0.2,line width=14,rounded corners=1] (18+72*4:1.31)--(18+72*2:1.31)--(18+72*2:1.29)--(18.01+72*4:1.29)--(18.01+72*4:3)--(17.99+72*4:3)--(17.99+72*4:1.3);
    \draw[BrickRed] (18+72*0:3) node[above=.2cm] {\Large $1$};
    \draw[BrickRed] (18+72*1:3) node[right=.2cm] {\Large $4$};
    \draw[BrickRed] (18+72*2:3) node[above=.2cm] {\Large $0$};
    \draw[BrickRed] (18+72*3:3) node[left=.2cm] {\Large $2$};
    \draw[BrickRed] (18+72*4:3) node[right=.2cm] {\Large $2$};
    \draw[BrickRed] (18+72*0:1.3) node[below=.1cm] {\Large $1$};
    \draw[BrickRed] (18+72*1:1.3) node[above right=.1cm] {\Large $0$};
    \draw[BrickRed] (18+72*2:1.3) node[above=.2cm] {\Large $3$};
    \draw[BrickRed] (18+72*3:1.3) node[above left] {\Large $3$};
    \draw[BrickRed] (18+72*4:1.3) node[below=.1cm] {\Large $2$};
    %\draw[BrickRed] (18+72*.4:3.7) node {$W_{10}(3,4,5,5,6)^{\trop}$};
    %\draw[BrickRed] (18+72*.4:3.5) node {$W_{10}(3,4,4,6,6)^{\trop}$};
\end{tikzpicture}
\caption{The balancing relations $B^5_{\{i,j\}}=4$ (left picture) and $B^1_{\{i,j\}}=7$ (right picture) for the tropicalization of $W_{10}(3,4,5,5,6)$. In each picture, the equally colored collections of vertices  all have the same sum. Example \ref{ex:M05RSpin} confirms that the sums of the light-green and light-purple collections of vertices in the left picture are both equal to 4. Note that we have named each ray $\mathbf v_I$ of $\M_{0,5}^{\trop}$ (corresponding to a vertex of the Petersen graph) using the convention $k\not\in I$---this is the only reason  the vertex labels of the two pictures differ.}
\label{fig:34556}
\end{figure}

\end{ex}

\subsubsection*{WDVV relations.} In $\overline{M}_{0,4}\cong \mathbb{P}^1$ we have the simple fact that any boundary divisor is equivalent in the Chow ring to the class of a point. Choosing any two of the three boundary divisors and subtracting them one obtains a relation, which we call a basic relation. Any relation obtained by pulling back via forgetful morphisms or pushing forward via gluing morphisms of a basic relation is called a WDVV relation.
Given a CohFT $\Omega$ on a vector space $V$, restricting a class $\Omega_{0,n}(v_1, \ldots, v_n)$ to a WDVV relation and integrating produces a relation among the numerical CohFT invariants, that we also call a WDVV relation. 
Explicitly, the relation obtained from pulling back a basic relation via some forgetful morphism has the following form. Denote by $\{e_\alpha\}$ a basis for $V$ and by  $\{e^\alpha\}$ the dual basis with respect to the CohFT metric. Choose four numbers $a,b,c,d$ in the index set $[n]$ and for a set $I$ denote  by ${\vec{v}_I} = \{v_i\}_{i\in I}$:
\begin{equation}\begin{aligned}\label{eq:wdvvrel}
\sum_{\alpha}\sum_{I\subseteq \{[n]\smallsetminus \{a,b,c,d\}\}}&\left(
\omega_{0, |I|+3}(v_a, v_b, \vec{v}_I, e_\alpha)\omega_{0, |I^c|+3}(e^\alpha, \vec{v}_{I^c}, v_c, v_d) \right. \\
& \left. - \omega_{0, |I|+3}(v_a, v_d, \vec{v}_I, e_\alpha)\omega_{0, |I^c|+3}(e^\alpha, \vec{v}_{I^c}, v_b, v_c)\right)= 0.
\end{aligned}\end{equation}

We now see that the numerical part of the WDVV relations have a natural analogue in the tropicalization.

\begin{thm}\label{thm:wdvv}
For any CohFT $\Omega$, consider the numerical CohFT invariants as unknowns to be determined; the following two collections of equations impose equivalent constraints among the invariants of $\omega$:
\begin{enumerate}
\item the balancing equations \eqref{eq:boomerang} for all tropical cycles $\Omega^{\trop}_{0,n}(v_1, \ldots, v_n)$. 
\item the collection of all WDVV relations.
\end{enumerate}
\end{thm}

\begin{proof}[Proof of Theorem \ref{thm:wdvv}]
We prove this statement in two steps. First we show that for a CohFT class of dimension $1$, the fact that its tropicalization is balanced is equivalent to numerical WDVV relations that are pull-backs of the basic ones via forgetful morphisms (and so take the simple form from equation \eqref{eq:wdvvrel}). Next we show that for a CohFT class of arbitrary dimension, its balancing along a cone is equivalent to WDVV relations that also involve pushforwards with respect to gluing morphisms.

\medskip

\noindent \textsc{Step I.} Assume $\dim \Omega_{0,n}(v_1,\ldots, v_n) = 1$. Its tropicalization can be written as a linear combination of rays:
\begin{equation}
   \Omega^{\trop}_{0,n}(v_1,\ldots, v_n) = \sum_{\substack{a\in I\subseteq [n]\smallsetminus k\\ 2\leq|I|\leq n-2}} a_I \langle {\mathbf{v}_I} \rangle_{\mathbb{R}_{\geq 0}}\subseteq Q_{[n]},
\end{equation}
with
\begin{equation}\label{eq:defcoeff}
a_I = \sum_{\alpha}
\omega_{0, |I|+1}(\vec{v}_I, e_\alpha)\omega_{0, |I^c|+1}(e^\alpha, \vec{v}_{I^c}),
\end{equation}
where $\omega_{0, |I|+1}(\vec{v}_I, e_\alpha)$ and $\omega_{0, |I^c|+1}(e^\alpha, \vec{v}_{I^c})$ are numerical CohFT invariants.

Assume $ \Omega^{\trop}_{0,n}(v_1,\ldots, v_n)$ is balanced, choose $a,b,c,d \in [n]$ and observe that the linear  projection $Q_{[n]}\to Q_{\{a,b,c,d\}}$ restricts to the forgetful morphism  $F:M^{\trop}_{0,n}\to M^{\trop}_{0,4}$, where the marks that are remembered are $a,b,c,d$. A useful fact used in tropical intersection theory (see e.g. \cite[Sec. 4]{AllermannRau2010}) is that the pushforward of a weighted, balanced fan via a map of fans is also a weighted, balanced fan. We  apply this fact to the map $F$ above. Here, as $\Omega_{0,n}(v_1, \dots, v_n)$ is a sum of rays, its pushforward along $F$ is a sum of rays in $M_{0,4}^{\trop}$.  Since $M_{0,4}^{\trop}$ is an irreducible fan, the only balanced weight functions are constant on all three rays, hence $F_*( \Omega_{0,n}(v_1, \dots, v_n))$ is a constant times the sum of the three rays of $M_{0,4}^{\trop}$.

The coefficient of the ray $\langle \mathbf{v}_{\{a,b\}}\rangle_{\mathbb{R}_{\geq 0}}$ in $F_\ast(\Omega^{\trop}_{0,n}(v_1,\ldots, v_n))$ equals the sum of the coefficients $a_I$ where $a,b\in I$ and $c,d \in I^c$, and similarly the coefficient of the ray $\langle \mathbf{v}_{\{a,d\}}\rangle_{\mathbb{R}_{\geq 0}}$ in $F_\ast(\Omega^{\trop}_{0,n}(v_1,\ldots, v_n))$ equals the sum of the coefficients $a_I$ where $a,d\in I$ and $b,c \in I^c$.
Imposing the equality of these coefficients and using \eqref{eq:defcoeff}, one obtains the WDVV relation \eqref{eq:wdvvrel}. Thus we have proven that satisfying balancing implies satisfying WDVV. 

Conversely, assume that the WDVV relations are satisfied among the numerical invariants of $\Omega$. We want to show that $ \Omega^{\trop}_{0,n}(v_1,\ldots, v_n)$ is balanced. By construction, ${\mathbf{v}_I}$ is the primitive vector for  the ray $\langle {\mathbf{v}_I} \rangle_{\mathbb{R}_{\geq 0}}$, and therefore  this means showing that 
\begin{equation}\label{eq:onedimbal}
    \sum_{\substack{I\subseteq [n]\smallsetminus k \\ 2\leq|I|\leq n-2}} a_I  {\mathbf{v}_I} =0\in Q_{[n]}.
\end{equation}
Note we have fixed a $k \in [n]$ to not double count each ray. We may use \eqref{eq:relvr} to rewrite:
\begin{equation}\label{eq:redtogenset}
    \sum_{\substack{I\subseteq [n]\setminus\{k\} \\ 2\leq|I|\leq n-2}} a_I  {\mathbf{v}_I} = \sum_{\{i,j\} \in [n]\setminus\{k\}} \left( \sum_{\substack{\{i,j\} \subseteq I\subseteq [n]\smallsetminus \{k\}\\ |I|\leq n-2}} a_I
\right) \mathbf{r}_{i,j}^k.
\end{equation}
 Subtracting the coefficient of $\mathbf{r}_{a,b}^c$ and the coefficient of $\mathbf{r}_{a,d}^c$ 
 in \eqref{eq:redtogenset} one obtains the WDVV relation \eqref{eq:wdvvrel}, which we are assuming holds. Since this happens for all choices of $a,b,c,d\in [n]$, all the coefficients in \eqref{eq:redtogenset} are equal and so \eqref{eq:onedimbal} holds, concluding the first part of the proof.

 \medskip
 
\noindent \textsc{Step II.} Let $\Omega_{0,n}(v_1, \ldots, v_n)$ be a CohFT class of pure dimension $d$, and $\tau$ a $(d-1)$-dimensional cone in $M_{0,n}^{\trop}$ corresponding to a graph $\Gamma$. The corresponding stratum $\Delta_\tau$ is isomorphic to a product $\prod_{{\rm v} \in V(\Gamma)} \overline{M}_{0, n_{\rm v}}$, where $n_{\rm v}$ denotes the valence of the vertex ${\rm v}$ of $\Gamma$. Denote by $\phi: \bigcup_{{\rm v}\in V(\Gamma)} \{{\rm v}\}\times [n_{\rm v}]\to \mathcal{P}([n])$ the function that assigns to each half-edge $h$ of $\Gamma$ the set of marks in the connected component of $\Gamma\smallsetminus v$ that contains $h$.

Choose a $k\in [n]$ and for every vertex ${\rm v}$ denote by $k_{\rm v}$ the unique element of $[n_{\rm v}]$ such that $k\in \phi([k_{\rm v}])$.

\begin{claim}\label{claim:comb}
  For every $I_{\rm v}\subseteq [n_{\rm v}]$ not containing $k_{\rm v}$, the assignment 
\begin{equation}
    \mathbf{v}_{I_{\rm v}}\mapsto \mathbf{v}_{\cup_{i\in I_{\rm v}} \phi(\rm v,i)}
\end{equation}
gives a linear injection 
\begin{equation}
   \Phi: \bigoplus_{{\rm v}\in V(\Gamma)} Q_{[n_{\rm v}]} \to Q_{[n]}/\langle \tau\rangle_{\mathbb{R}}.
\end{equation}  
\end{claim}

We defer the proof of the claim to not break the flow of this proof. For dimension reasons, the CohFT class of dimension $d$ restricted to the stratum associated to the cone $\tau$ can be  nonzero only if it decomposes as a sum over the vertices ${\rm v}$ of $\Gamma$ where each summand is a product of a one-dimensional CohFT class $\Omega^1_{\rm v}$ supported on the vertex $ {\rm v}$ with a multiple of the class of a point on all the other vertices.  In formulas, we have:
\begin{equation}\label{eq:cvs in formula}
   \Omega_{0,n}(v_1, \ldots, v_n)|_{\Delta_\tau} = 
\sum_{{\rm v}\in V(\Gamma)} {\gl_\Gamma}_\ast \left(c_{\rm v}\pi_{\rm v}^\ast(\Omega^1_{\rm v})\boxtimes \prod_{\rm w\not= v} \pi_{\rm w}^\ast(pt)
\right), 
\end{equation}
where $\gl_\Gamma: \prod_{{\rm v} \in V(\Gamma)} \overline{M}_{0, n_{\rm v}}\to\overline{M}_{0, n} $ is the gluing morphism, $\pi_{\rm v}$ denotes  the projection on the factor corresponding to vertex ${\rm v}$ and the coefficients $c_{\rm v}$ are real numbers determined by the splitting axioms. These coefficients are obtained from numerical CohFT invariants and the CohFT metric, but this information will not be explicitly used in this argument.

Any  $d$-dimensional cone $\sigma$ in $M^{\trop}_{0,n}$ of which $\tau$ is a face
corresponds to a graph $\Gamma_\sigma$ with an edge contraction $c_{e_\sigma}:\Gamma_\sigma\to \Gamma$.
Denote by $\hat{\rm v}:\{\sigma \succ \tau \}\to V(\Gamma)$ the function that assigns to each $d$-dimensional cone $\sigma$ the vertex $c_{e_\sigma}(e_\sigma)$ of $\Gamma$ that is the image of the contracted edge. The edge $e_\sigma$ also determines a two part partition of the set of indices in $[n_{\hat{\rm v}(\sigma)}]$, i.e. a ray $\rho_\sigma$ with primitive vector $\mathbf{v}_{I_\sigma}$ in $M^{\trop}_{0,n_{\hat{\rm v}}(\sigma)} \subseteq Q_{[n_{\hat{\rm v}(\sigma)}]}$.
By the CohFT splitting axioms, the restriction of $\Omega_{0,n}(v_1, \ldots, v_n)$ to $\Delta_\sigma$ satisfies
\begin{equation}
   \int_{\Delta_\sigma}\Omega_{0,n}(v_1, \ldots, v_n) = c_{\rm v}\ \int_{\overline{M}_{0, n_{{\rm \hat{v} 
 (\sigma)}}}}\Omega^1_{\rm v}\cdot \Delta_{\rho_\sigma}. 
\end{equation}
We can now analyze the balancing condition along the face $\tau$:
\begin{equation}\label{eq:batau}
    \sum_{\sigma \succ \tau} \Omega_{0,n}(v_1, \ldots, v_n)|_{\Delta_\sigma} {\mathbf{u}_{\tau/\sigma}} = \sum_{{\rm v}\in V(\Gamma)} c_{\rm v} \Phi\left(\sum_{\sigma \in \hat{\rm v}^{-1}({\rm v})}
    \left(
    \int_{\overline{M}_{0,n_{\hat{\rm v}}(\sigma)}}\Omega^1_{\rm v}\cdot \Delta_{\rho_\sigma}
\right)    \mathbf{v}_{I_\sigma}
\right).
\end{equation}
The vanishing of the left hand side of \eqref{eq:batau} gives the balancing of $\Omega_{0,n}(v_1, \ldots, v_n)$ along the higher dimensional face $\tau$; on the right hand side this quantity is expressed as a sum over the vertices $\rm v$ of $\Gamma$; inside the large parenthesis, we recognize the weighted sum of normal vectors whose vanishing gives balancing  for the one dimensional CohFT class $\Omega^1_{\rm v}$ at the cone point of $\overline{M}_{0,n_{\hat{\rm v}}(\sigma)}$. Then we take the image via $\Phi$, which is a linear injection by Claim \ref{claim:comb}. Thus \eqref{eq:batau} shows that the balancing equation for a CohFT class along a higher dimensional face $\tau$ reduces  to a collection of balancing equations for one-dimensional CohFT classes. The theorem is proved by invoking Step I. 
\end{proof}

\begin{proof}[Proof of Claim \ref{claim:comb}]
This statement is well-known to the experts, but we include a proof here since we are not aware of a reference in the literature. We begin by proving the codimension $1$ case. Let $\tau$ be a ray in $M_{0,n}^{\trop}$ corresponding to a graph $\Gamma$ with only one edge separating the legs into two subsets $I$ and $I^c$. Without loss of generality, we assume that $k\not\in I$ and that the marks in $I$ are adjacent to ${\rm v}_1$, those in $I^c$ to ${\rm v_2}$. In this case, $n_{\rm v_1} = |I|+1$, and we denote the additional mark (corresponding to the germ of the edge) by $k_{\rm v_1}$, following the notation in the proof of Theorem \ref{thm:wdvv}.
The one dimensional linear space of $Q_{[n]}$ containing the ray $\tau$ is spanned by 
the vector $\mathbf{v}_\tau = \sum_{i,j\in I} \mathbf{r}_{i,j}^k$.
The map $\Phi|_{Q_{[n_{{\rm v}_1}]}}$ is a linear isomorphism onto its image in $Q_{[n]}/\langle\tau\rangle_{\mathbb{R}}$: it is defined by $\mathbf{r}_{i,j}^{k_{{\rm v}_1}}\mapsto \mathbf{r}_{i,j}^k$ and the image of the relation among the generators of $Q_{[n_{{\rm v}_1}]}$ is precisely the vector $\mathbf{v}_\tau$. Denoting by $x$ the mark corresponding to the germ of the edge attaching to ${\rm v_2}$, the map $\Phi|_{Q_{[n_{{\rm v}_2}]}}$ is defined as follows:
\begin{align*}
 \mathbf{r}_{i,j}^{k}\mapsto \mathbf{r}_{i,j}^k\ \ \ \ \    &  \ \ \ \ \ 
 \mbox{if $i,j$ in $I^c$}\\
  \mathbf{r}_{j,x}^{k}\mapsto \sum_{i\in I}\mathbf{r}_{i,j}^k & \ \ \ \ \  \mbox{if $j$ in $I^c$},
\end{align*}
where we have already set $\mathbf{v}_\tau = 0$. It is immediate  to check that the restriction $\Phi|_{Q_{[n_{{\rm v}_2}]}}$ is injective and that $\Phi{(Q_{[n_{{\rm v}_1}]}\oplus 0)}\cap \Phi{(0\oplus Q_{[n_{{\rm v}_1}]})} = 0$, proving the claim in the case $\dim \tau = 1$. To prove the claim for a  general cone $\tau$ of dimension $d$ (corresponding to a graph $\Gamma$) one can choose arbitrarily an ordering of the edges of $\Gamma$, and iterate this construction $d$ times, adding one edge at a time  until $\Gamma$ is obtained, and composing all the resulting linear injections.  
\end{proof}

\section{Recursive structure of $r$-spin invariants}
\label{sec:recursivestructure}

In this section we introduce a   recursion for $r$-spin invariants, deducible from either the WDVV equations or (equivalently by Theorem \ref{thm:TropicalCycle}) the balancing relations of tropical cycles. We will use the recursion extensively in proving Theorem \ref{thm:ClosedFormula} and \ref{thm:RespectsOrdering}.

\begin{thm}\label{thm:Recursion}
Let $\vec m=(m_1,\ldots,m_n)$ be a %n $n$-pointed
 numerical $r$-spin monodromy vector. % with $\sum_{i=1}^nm_i=(n-2)(r+1)$. 
 Fix distinct elements $i,j,k\in[n]$ with $m_i>1$ and $m_j<r$. Then
\begin{align}\label{eq:Recursion}
    w_r(\vec m)-w_r(m_1,\ldots,m_{i}-1,\ldots,m_j+1,\ldots,m_n)=(n-3)\cdot T_r^{i,j,k}(m_1,\ldots,m_n)
\end{align}
where, for distinct elements $i,j,k\in [n]$, 
\begin{align}\label{eq:TIJK}
    T_r^{i,j,k}(m_1,&\ldots,m_n):=\frac{1}{r}\Bigl(\delta_{m_j+m_k\ge r+1}w_r(m_i-1,m_1,\ldots,\hat{m_i},\hat{m_j},\hat{m_k},\ldots,m_{n},m_j+m_k-r)\nonumber\\
    &-\delta_{m_i+m_k\ge r+2}w_r(m_j,m_1,\ldots,\hat{m_i},\hat{m_j},\hat{m_k},\ldots,m_{n},m_i+m_k-r-1)\Bigr);
\end{align}
 the symbol $\delta_{a\geq b}$ equals one when the inequality is satisfied and zero otherwise.
\end{thm}
\begin{remark}
    The left side of \eqref{eq:Recursion} is independent of $k,$ which implies that $T_r^{i,j,k}(m_1,\ldots,m_n)$ is independent of $k$.
\end{remark}

\begin{proof}
    Since $w_r(\vec m)$ is invariant under permuting the entries of $\vec m$, we may assume $i=1$ and $j=2.$ By Theorem \ref{thm:wdvv}, we may equivalently use either the WDVV equations or the balancing condition for tropical cycles to carry out the proof. Let $$\vec m^{*}=(m_1^*,\ldots,m_{n+1}^*)=(m_1-1,m_2,m_3,\ldots,m_n,2),$$ and consider the 1-dimensional tropical cycle $W_r(\vec m^{*})$. The balancing condition implies
\begin{align*}
    B^{n+1}_{\{1,3\}}(W_r(\vec m^{*}))=\sum_{\substack{J\subseteq\{1,\ldots,n+1\}\\2\le\abs{J}\le n-1\\1,3\in J,\thickspace n+1\not\in J}}\int_{\Mbar_{0,n}}W_r(\vec m^{*})|_{D_J}=\sum_{\substack{J\subseteq\{1,\ldots,n+1\}\\2\le\abs{J}\le n-1\\2,3\in J,\thickspace n+1\not\in J}}\int_{\Mbar_{0,n}}W_r(\vec m^{*})|_{D_J}=B^{n+1}_{\{2,3\}}(W_r(\vec m^{*})).
\end{align*}
Terms $W_r(\vec m^{*})|_{D_J}$ with $\{1,2,3\}\subseteq J$ appear on both sides of the above equation, so we may remove them, obtaining a WDVV equation:
\begin{align}\label{eq:SpecialBoomerang}
    \sum_{\substack{J\subseteq\{1,\ldots,n+1\}\\1,3\in J,\thickspace2,n+1\not\in J}}\int_{\Mbar_{0,n}}W_r(\vec m^{*})|_{D_J}&=\sum_{\substack{J\subseteq\{1,\ldots,n+1\}\\2,3\in J,\thickspace1,n+1\not\in J}}\int_{\Mbar_{0,n}}W_r(\vec m^{*})|_{D_J}.
\end{align}
By the CohFT property (Proposition \ref{ax:RestrictToStratum}), we have:
\begin{align}\label{eq:CohFTProperty}
\int_{\Mbar_{0,n}}W_r(\vec m^*)|_{D_J}=
w_r(( m_i^*)_{i\in J},m_J^*))w_r((m_i^*)_{i\not\in J},m_{J^c}^*),
\end{align}
where $m_J,m_{J^c}$ are as in \eqref{monodromy at nodes}. Note that by Proposition \ref{prop:1DVanish}, \eqref{eq:CohFTProperty} vanishes unless $$(\abs{J}-1)(r+1)-r<\sum_{i\in J}m_i^*<(\abs{J}-1)(r+1).$$
We now apply the following additional vanishing property, which we prove after finishing this proof.
\begin{claim}\label{cl:Insertion2Vanishing}
Let $n\ge 5$ and suppose $\vec m=(m_1,\ldots,m_n)$ is a numerical $r$-spin monodromy vector with $m_i=2$ for some $i$. Then $w_r(\vec m)=0$.
\end{claim}

Since $m_{n+1} = 2$, Claim \ref{cl:Insertion2Vanishing} and \eqref{eq:CohFTProperty} imply that any term of either side of \eqref{eq:SpecialBoomerang} with $|J^c|>4$, 
 or equivalently $\abs{J} < n-3$,  vanishes. The remaining terms on the left are: 
\begin{itemize}
\item The term with $J=[n+1]\setminus\{2,n+1\}$, i.e. $$w_r(m_1-1,m_3,\ldots,m_n,m_2+1)w_r(m_2,2,r-1-m_2)=w_r(m_1-1,m_2+1,m_3,\ldots,m_n),$$ and
\item The terms with $J=[n+1]\setminus\{2,k,n+1\}$ for $k\in\{4,\ldots,n\},$ namely \begin{align*}\sum_{k=4}^nw_r(m_1-1,m_3,m_4,\ldots,&\hat m_k,\ldots,m_n,m_2+m_k-r)w_r(m_2,m_k,2,2r-m_2-m_k)\\&=\frac{1}{r}\sum_{k=4}^n\delta_{m_2+m_k\ge r+1}\cdot w_r(m_1-1,m_3,\ldots,\hat m_k,\ldots,m_n,m_2+m_k-r).\end{align*}
\end{itemize} 
Identical calculations show that the right side of \eqref{eq:SpecialBoomerang} is $$w_r(\vec m)+\frac{1}{r}\sum_{k=4}^n\delta_{m_1+m_k\ge r+2}\cdot w_r(m_2,m_3,\ldots,\hat m_k,\ldots,m_n,m_1+m_k-r-1).$$ Combining all terms above yields \begin{align}\label{eq:Presymmetry}
    w_r(\vec m)-w_r(m_1-1,m_2+1,m_3,\ldots,m_n)&=\sum_{k=4}^nT^{1,2,k}(\vec m).
\end{align}
Finally, we observe that the left side of \eqref{eq:Presymmetry} is invariant under permuting $m_3,\ldots,m_n$, hence the right side is also. It follows that the $n-3$ summands on the right side of \eqref{eq:Presymmetry} are all equal, yielding \eqref{eq:Recursion}.
\end{proof}
\begin{proof}[Proof of Claim \ref{cl:Insertion2Vanishing}]
   If $m_i=r$ for some $i\in[n],$ then $w_r(\vec m)=0,$ so assume $m_i\le r-1$ for all $i\in[n]$. Suppose $m_1=2.$ Define the monodromy vector 
    $$\vec m^{**}=(m_1^{**},\ldots,m_{n+1}^{**})=(2,m_2-1,m_3,\ldots,m_n,2).$$ 
    Analogously to \eqref{eq:SpecialBoomerang}, the balancing condition $B^{3}_{\{1,2\}}(W_r(\vec m^{**}))=B^3_{\{1,4\}}(W_r(\vec m^{**}))$ for the 1-dimensional tropical cycle $W_r(\vec m^{**})$ implies the WDVV equation \begin{align}\label{eq:WDVVInsertion2}
    \sum_{\substack{J\subseteq\{1,\ldots,n+1\}\\1,2\in J,\thickspace3,4\not\in J}}\int_{\Mbar_{0,n}}W_r(\vec m^{**})|_{D_J}=\sum_{\substack{J\subseteq\{1,\ldots,n+1\}\\1,4\in J,\thickspace2,3\not\in J}}\int_{\Mbar_{0,n}}W_r(\vec m^{**})|_{D_J}.
    \end{align}
   
    First note that using the constraints $m_i \le r-1$, $m_1=2$, $n \ge 5$, and $\sum_{i=1}^{n} m_i = (n-2)(r+1)$, one can deduce that:
    \begin{align*}\label{eq:FirstObs}
        \text{$m_i+m_j\ge r+3$ for all distinct $i,j\in\{2,\ldots,n\}$}.\tag{$A$}
    \end{align*}
    Second, using $m_1^{**}=m_{n+1}^{**}=2$, $m_i^{**} \le r-1$, and $\sum_{i=1}^{n+1}m_i=(n-1)(r+1)-r$ we have that if $1,n+1 \in J$ then:
    $$
    \sum_{i\in {J^c}}m_i\geq(n-1)(r+1)-r-4-(\abs{J}-2)(r-1)=(\abs{J^c}-1)(r+1)+(2\abs{J}-7).
    $$
 Proposition~\ref{prop:1DVanish} implies $W_r(\vec m^{**})|_{D_J}=0$ if $\abs{J}\geq 4$. If $\abs{J}=3$, then $W_r(\vec m^{**})|_{D_J}$ could be nonzero only if $m_{J^c}=1$. But in this case $n\ge5$ implies $\abs{J^c}\ge3$, so Propositions~\ref{ax:RestrictToStratum} and \ref{prop:Insertion1Vanish} imply $W_r(\vec m^{**})|_{D_J}=0$. To summarize:
    \begin{align*}
    \text{If $\abs{J}\ge3$ and $1, n+1\in J$, then $\int_{\Mbar_{0,n}}W_r(\vec m^{**})|_{D_J}=0$.}\tag{$B$}\label{eq:SecondObs}
    \end{align*}
We now show $w_r(\vec m)=0$ by induction on $n$, with base case $n=5$. If $n=5$, then a straightforward analysis using Proposition~\ref{ax:RestrictToStratum}, Proposition~\ref{prop:34Point}, and Observations \eqref{eq:FirstObs} and \eqref{eq:SecondObs} imply the following about the terms in the left side of \eqref{eq:WDVVInsertion2}: 
\begin{itemize}
    \item  When $J=\{1,2\}$, we have $\int_{\Mbar_{0,n}}W_r(\vec m^{**})|_{D_J}=w_r(2, m_2-1, r-m_2) \cdot w_r(\vec m)=w_r(\vec m).$
    \item When $J=\{1,2,5\}$, we use \eqref{eq:FirstObs} to see that $\int_{\Mbar_{0,n}}W_r(\vec m^{**})|_{D_J}= (1/r^2)\cdot\delta_{m_2+m_5\ge r+3}=1/r^2$.
    \item When $J=\{1,2,6\}$,  we use \eqref{eq:SecondObs} to see that $\int_{\Mbar_{0,n}}W_r(\vec m^{**})|_{D_J}=0$. 
    \item  When $J=\{1,2,5,6\}$, we use \eqref{eq:SecondObs} to see that $\int_{\Mbar_{0,n}}W_r(\vec m^{**})|_{D_J}=0$.
\end{itemize}
By the same arguments, the right side of \eqref{eq:WDVVInsertion2} has the following terms:
    \begin{itemize}
        \item When $J=\{1,4\}$, we have $\int_{\Mbar_{0,n}}W_r(\vec m^{**})|_{D_J}=w_r(2,m_2-1,m_3,m_4+1,m_5)$.
        \item When $J=\{1,4,5\}$, we use \eqref{eq:FirstObs} to see that $\int_{\Mbar_{0,n}}W_r(\vec m^{**})|_{D_J}=(1/r^2)\cdot\delta_{m_2+m_3\ge r+2}=1/r^2$.
        \item When $J=\{1,4,6\}$, we use \eqref{eq:SecondObs} to see that $\int_{\Mbar_{0,n}}W_r(\vec m^{**})|_{D_J}=0$. 
        \item When $J=\{1,4,5,6\}$, we use \eqref{eq:SecondObs} 
        to see that $\int_{\Mbar_{0,n}}W_r(\vec m^{**})|_{D_J}=0$.
    \end{itemize}

    Thus \eqref{eq:WDVVInsertion2} reduces to
    $$w_r(\vec m)=w_r(2,m_2-1,m_3,m_4+1,m_5).$$
    Iterating, we have $$w_r(\vec m)=w_r(m_1,m_2-j,m_3,m_4+j,m_5)$$ for all $0\le j\le \min(m_2-1, r-m_4)$. In fact \eqref{eq:FirstObs} implies $\min(m_2-1,r-m_4)=r-m_4.$ Taking $j= r-m_4$, we conclude $w_r(\vec m)=0$ from Proposition \ref{prop:RamondVanishing}. This completes the base case.

    Now assume $n>5$ and that the claim holds for all $5\le n'<n$. The inductive hypothesis implies that every term on both sides of \eqref{eq:WDVVInsertion2} with $4\le\abs{J}\le n-2$ vanishes. We now account for the remaining terms on the left side of \eqref{eq:WDVVInsertion2}: 
    \begin{itemize}
        \item When $J=\{1,2\}$, we have $\int_{\Mbar_{0,n}}W_r(\vec m^{**})|_{D_J}=w_r(\vec m)$.
        \item When $J=\{1,2,i\}$ for $i\in\{5,\ldots,n\}$, the inductive hypothesis shows $\int_{\Mbar_{0,n}}W_r(\vec m^{**})|_{D_J}=0$.
        \item When $J = \{1,2,n+1\}$,  we use \eqref{eq:SecondObs} to see that $\int_{\Mbar_{0,n}}W_r(\vec m^{**})|_{D_J}=0$.
        \item When $J=[n]\setminus\{3,4\}$, we use \eqref{eq:SecondObs} to see that $\int_{\Mbar_{0,n}}W_r(\vec m^{**})|_{D_J}=0$.
    \end{itemize}
    We then account for the right side:
    \begin{itemize}
        \item When $J=\{1,4\}$, we have that $\int_{\Mbar_{0,n}}W_r(\vec m^{**})|_{D_J}=w_r(2,m_2-1,m_3,m_4+1,m_5,\ldots,m_n)$.
        \item When $J=\{1,4,i\}$ for $i\in\{5,\ldots,n\}$, the inductive hypothesis shows $\int_{\Mbar_{0,n}}W_r(\vec m^{**})|_{D_J}=0$.
        \item When $J=\{1,4,n+1\}$, we use \eqref{eq:SecondObs} to see that $\int_{\Mbar_{0,n}}W_r(\vec m^{**})|_{D_J}=0$.
        \item When $J=[n+1]\setminus\{2,3\}$, we use \eqref{eq:SecondObs} to see that $\int_{\Mbar_{0,n}}W_r(\vec m^{**})|_{D_J}=0$.
    \end{itemize}
     Thus \eqref{eq:WDVVInsertion2} reads $$w_r(\vec m)=w_r(2,m_2-1,m_3,m_4+1,m_5,\ldots,m_n).$$ As in the base case, we iterate to see $w_r(\vec m)=w_r(2,m_2-j,m_3,m_4+j,m_5,\ldots,m_n)$ for $0 \le j \le \min(m_2-1, r-m_4)$. Again, $\min(m_2-1,r-m_4)=r-m_4$ by \eqref{eq:FirstObs}. Taking  $j = r-m_4$, we obtain $w_r(\vec m)=0$.
    \end{proof}

We next use Theorem \ref{thm:Recursion} to give direct proofs of two statements proved by Pandharipande-Pixton-Zvonkine \cite{PPZ2019} using the representation theory of $\mathfrak{sl}_2(\C)$. The first, which follows from \cite[Thm. 2]{PPZ2019}, is the following divisibility statement.
\begin{cor}
 For any numerical monodromy vector $\vec m$ of length $n$,   $w_r(\vec m)$ is an integer multiple of $\frac{(n-3)!}{r^{n-3}}$.
\end{cor}
\begin{proof}
    We induct on $n$. The base case $n=3$ follows from \ref{prop:34Point}. If $n>3,$ \eqref{eq:Recursion} expresses $w_r(\vec m)$ as $\frac{n-3}{r}\cdot F$, where $F$ is a $\Z$-linear combination in invariants $w_r(\vec m')$ with $\abs{m'}=n-1.$ By the inductive hypothesis, $F$ is an integer multiple of $\frac{(n-4)!}{r^{n-4}}$, and the result follows.
\end{proof}
In Theorem~\ref{thm:RespectsOrdering}, we will see that in fact $w_r(\vec m)$ is a \emph{nonnegative} integer multiple of $\frac{(n-3)!}{r^{n-3}}$. The second statement is the following generalization of Claim \ref{cl:Insertion2Vanishing}, which is \cite[Prop. 1.4]{PPZ2019}.
\begin{cor}\label{cor:RSpinVanishing}
Let $n\ge 4,$ and suppose $\vec m=(m_1,\ldots,m_n)$ is an $r$-spin monodromy vector with $\sum_{i=1}^n m_i=(n-2)(r+1)$. If $m_i\le n-3$ for some $i$, then $w_r(\vec m)=0$.
\end{cor}
\begin{proof}
    We may assume $m_1\le n-3.$ We proceed by induction on $n$ and $m_1$, with base cases $m_1=1$ (for all $n$) following from Proposition \ref{prop:34Point}.
    
    If $n=4$ the base case is the only case, so let $n>4.$ Assume that the Corollary holds for $n'<n$ and for $n' = n, \  m_1'<m_1$. Since $\sum_im_i=(n-2)(r+1),$ we may also assume without loss of generality that $m_2<r.$ By Theorem \ref{thm:Recursion}, $$w_r(\vec m)-w_r(m_1-1,m_2+1,m_3,\ldots,m_n)=(n-3)T^{1,2,3}_r(m_1,\ldots,m_n).$$ Both terms of $T^{1,2,3}_r(m_1,\ldots,m_n)$ contain an $(n-1)$-pointed $r$-spin invariant. By induction on $n$, both of these invariants are zero; the first term has an insertion $m_1-1\le(n-1)-3$, and the second term has an insertion $m_1+m_3-r-1\le m_1-1\le(n-1)-3$. Thus $$w_r(\vec m)=w_r(m_1-1,m_2+1,m_3,\ldots,m_n).$$ By induction on $m_1$, $w_r(\vec m)=0.$
\end{proof}

Our next goal is to prove Lemma \ref{lem:Reconstruct}, a reconstruction result for $w_r(\vec m)$, which we will use in Section \ref{sec:ClosedFormula} to give a recursive proof of Theorem \ref{thm:ClosedFormula}.

\begin{definition}
Let $\vec{m} = (m_i)$ and $\vec{m}' = (m_i')$ be two monodromy vectors of length $n$. We say that $\vec{m}$ and $\vec{m}'$ are \emph{neighbors} if there are $i, j \in \{1, \dots, n\}$ so that
$$
\vec{m} - \vec{m}' = (d_1, \dots, d_n), \text{ where } d_k = \delta_{ik} - \delta_{jk},
$$
or, equivalently, $\vec m - \vec m'$ is some permutation of the vector $(-1,1,0,\ldots, 0)$.
\end{definition}

\begin{ex}
The two $r$-spin monodromy vectors $(2,2,r-1,r-1)$ and $(2,3,r-2,r-1)$ are neighbors.
\end{ex}
We need the following fact.
\begin{lemma}\label{lem:sequence of neighbors}
Given two $r$-spin monodromy vectors $\vec m$ and $\vec {m'}$ of length $n$, there exists a sequence of monodromy vectors $\vec {m_0} = \vec m, \vec{m_1}, \ldots, \vec{m}_{n-1}, \vec{m_n} = \vec{m'}$ such that $\vec{m_i}$ and $\vec{m}_{i+1}$ are neighbors for all $i$.
\end{lemma}
\begin{proof}
This follows from a standard inductive argument.
\end{proof}

%The following lemma shows that genus-zero $r$-spin invariants can be reconstructed using Theorem \ref{thm:Recursion}.
\begin{lem}\label{lem:Reconstruct}
    Let $D$ be the set of numerical monodromy vectors for a given $r$, and let $\widetilde w_r:D\to\C$ be a function such that:
    \begin{enumerate}[label=(\roman*)]
    \item $\widetilde w_r(\vec m) = 1$ when $|\vec m| = 3$,\label{it:3PointCondition}
    \item $\widetilde w_r(\vec m)=0$ if $\abs{\vec m}\ge4$ and $m_i=r$ for some $i\in\{1,\ldots,\abs{\vec m}\},$\label{it:VanishingCondition}
    \item $\widetilde w_r$ satisfies the recursion \eqref{eq:Recursion}.\label{it:recursion}
    \end{enumerate}
    Then $\widetilde w_r=w_r$.
\end{lem}
\begin{remark}
    Lemma~\ref{lem:Reconstruct} is a slight variation of previous reconstruction lemmas in the literature. In \cite[Prop. 6.2]{JKV01}, property \ref{it:recursion} is replaced by the WDVV relations and a nonvanishing 4-point invariant. In \cite[Lem. 1.3]{PPZ2019}, property \ref{it:VanishingCondition} is replaced with vanishing of $\widetilde w_r(\vec m)$ if $m_i=1$ for some $i$. Lemma~\ref{lem:Reconstruct} is in a form particularly suited to proving Theorem~\ref{thm:ClosedFormula}. 
\end{remark}
\begin{proof}[Proof of Lemma \ref{lem:Reconstruct}]
We prove that $\widetilde w_r=w_r$ by induction on the length $n$ of the monodromy vector. The base case $n=3$ holds by Proposition~\ref{prop:34Point}. 
Let $k>3$ and assume $w_r=\widetilde w_r$ holds for all monodromy vectors of length less than $k$. By \ref{it:VanishingCondition} and Proposition~\ref{prop:RamondVanishing}, $\widetilde w_r(\vec m) = w_r(\vec m)$ if $\vec m$ is a length-$k$ monodromy vector with $m_i=r$ for some $i$. (It is easy to see that there must exist such a monodromy vector.) 

By Theorem~\ref{thm:Recursion}, if $\vec m$ and $\vec m'$ are neighbors such that $w_r(\vec m) = \widetilde w_r(\vec m)$, and such that $w_r(\vec{a}) = \widetilde w_r(\vec {a})$ for all monodromy vectors $\vec{a}$ of length less than $|\vec m|$, then $w_r(\vec m') = \widetilde w_r(\vec m')$. Combining this with Lemma~\ref{lem:sequence of neighbors} implies that $w_r(\vec m) = \widetilde w_r(\vec m)$ for all monodromy vectors $\vec m$ of length $k$, proving the claim.
\end{proof}

\section{Explicit piecewise-polynomial formula}\label{sec:ClosedFormula}

In this section, we prove a closed formula for genus-zero primary $r$-spin invariants. 
\begin{thm}\label{thm:ClosedFormula}
For $\vec m=(m_1,\ldots,m_n)$ a numerical  genus-zero $r$-spin monodromy vector, we have the formula
    \begin{align}\label{eq:ClosedFormula}
        w_r(\vec m)=\frac{1}{2r^{n-3}}\sum_{\substack{S\subseteq[n]\\\sum_{i\in S}m_i\ge(\abs{S}-1)r+n-2}}(-1)^{1+\abs{S}}\prod_{k=1}^{n-3}\Bigl(\Bigl(\sum_{i\in S}m_i\Bigr)-(\abs{S}-1)r-k\Bigr)
    \end{align}
\end{thm}
This immediately implies:
\begin{cor}\label{cor:PiecewisePolynomial}
    $w_r(m_1,\ldots,m_n)$ is a piecewise-polynomial function of the inputs $m_1,\ldots,m_n,r$. Precisely, the set $$\mathbf K_n:=\{(m_1,\ldots,m_n,r)\thickspace|\thickspace1\le m_i\le r-1,\textstyle\sum_im_i=(n-2)(r+1)\}\subseteq\Z_{>0}^n\times\Z_{>0}$$ of numerical $n$-pointed monodromy vectors is the set of lattice points of an unbounded polyhedron $\mathbf M_n$, and there is a finite wall-chamber decomposition of $\mathbf M_n$ such that in each chamber, $w_r(m_1,\ldots,m_n)$ is a polynomial.
\end{cor}
\begin{proof}
    This is immediate from the form of \eqref{eq:ClosedFormula}; the walls are the affine-linear spaces $$L_S:=\{(m_1,\ldots,m_n,r)\thickspace|\thickspace{\textstyle\sum_{i\in S}m_i}=(\abs{S}-1)r+n-3\}$$ for subsets $S\subseteq[n].$
\end{proof}
\begin{remark}\label{rem:ThickWalls}
    In fact, a stronger version of piecewise-polynomiality holds. The walls of the above wall-chamber decomposition have ``width $n-3$'' in the sense that if $E$ and $E'$ are adjacent chambers separated by a single wall $L_S$, and $P,P'\in\Q[m_1,\ldots,m_n,r]$ are such that $w_r(\vec m)=P$ after restriction to $E$ and $w_r(\vec m)=P'$ after restriction to $E'$, then $P$ and $P'$ agree not only on the wall $L_S$, but on the set $$\{(m_1,\ldots,m_n,r)\in\mathbf K_n\thickspace|\thickspace(\abs{S}-1)r+1\le\textstyle\sum_{i\in S}m_i\le(\abs{S}-1)r+n-3\}.$$
\end{remark}
For the proof of Theorem \ref{thm:ClosedFormula}, we introduce two pieces of notation.
\begin{notation}
     For $\vec m=(m_1,\ldots,m_n)$ a numerical $r$-spin monodromy vector, let $\widetilde w_r(\vec m)$ denote the right side of \eqref{eq:ClosedFormula}:
    \begin{align}\label{eq:WTildeDef}
   \widetilde w_r(\vec m) := \frac{1}{2r^{n-3}}\sum_{\substack{S\subseteq[n]\\\sum_{i\in S}m_i\ge(\abs{S}-1)r+n-2}}(-1)^{1+\abs{S}}\prod_{k=1}^{n-3}\Bigl(\Bigl(\sum_{i\in S}m_i\Bigr)-(\abs{S}-1)r-k\Bigr).
    \end{align}
\end{notation}
\noindent Therefore our goal for the remainder of this section is to prove $\widetilde w_r(\vec m)=w_r(\vec m)$ for all $r$ and $\vec m$.
\begin{notation}\label{star star star}
    Let $\vec m$ be a fixed $n$-pointed genus-zero $r$-spin monodromy vector. For a subset $S\subseteq[n]$, we say that $S$ \emph{satisfies} \eqref{eq:ConditionStar} (resp. \eqref{eq:ConditionStar2}) if \begin{align}\label{eq:ConditionStar}
        \sum_{i\in S}m_i&\ge(\abs{S}-1)r+n-2\tag{$\star1$}\\\sum_{i\in S}m_i&\ge(\abs{S}-1)r+n-3\tag{$\star2$}.\label{eq:ConditionStar2}
    \end{align}
    We will say e.g. $S$ \emph{satisfies \eqref{eq:ConditionStar} with respect to $\vec m$} if $\vec m$ is unclear from context. From \eqref{eq:WTildeDef}, it is clear why Condition \eqref{eq:ConditionStar} is useful---the usefulness of \eqref{eq:ConditionStar2} will become clear in the proof of Theorem \ref{thm:ClosedFormula}.
    
    Observe that if $\abs{S}\le1$, then $S$ necessarily satisfies \eqref{eq:ConditionStar}, while if $\abs{S}\ge n-1$, then $S$ cannot satisfy \eqref{eq:ConditionStar}.
\end{notation}

We now state several lemmas we will use in the proof, starting with the following well-known combinatorial fact.
\begin{lemma}\label{lem:AlternatingSumLemma}
    Let $T\subseteq A$ be finite sets, and let $F(x)\in\C[x]$ with $\deg(F(x))<\abs{A}-\abs{T}$. Then \begin{align}\label{eq:AlternatingSumLemma}
        \sum_{T\subseteq S\subseteq A}(-1)^{\abs{S}}F(\abs{S})=0.
    \end{align}
\end{lemma}
\begin{proof}
We reduce to the case $T=\emptyset$, by summing over $S\subseteq A\setminus T$ and performing the substitution $\abs{S}\mapsto\abs{S}+\abs{T}$ (which does not affect the degree of $F$). We then reduce to the case $F(x)=x^j$ since the left side of \eqref{eq:AlternatingSumLemma} is linear in $F$. Finally, the left side of \eqref{eq:AlternatingSumLemma} is equal to $\sum_{k=0}^{|A|} (-1)^k \binom{|A|}{k}k^j,
$
 which (up to sign) counts surjections $[j]\to A$ by a famous inclusion-exclusion argument. There are clearly no such surjections.
\end{proof}

 In proving Theorem \ref{thm:ClosedFormula}, we will apply Lemma \ref{lem:AlternatingSumLemma} in the following form. 
\begin{lemma}\label{lem:low degree vanishing}
    Let $n$ be a positive integer, let $m_1, \dots, m_{n-2}\in\C$, and let $\ell_k(x) \in \C[x]$ be linear functions for $k = 1, \dots, n-4$. Then
     \begin{align}\label{eq:low degree vanishing}
        \sum_{S\subseteq\{1,2,\ldots,n-2\}}(-1)^{1+\abs{S}}\prod_{k=1}^{n-4}\Bigl(\Bigl(\sum_{i\in S}m_i \Bigr)-\ell_k(|S|)\Bigr)=0,
     \end{align}
\end{lemma}
    \begin{proof} 
     Consider the left side of \eqref{eq:low degree vanishing} as a degree-$(n-4)$ polynomial in \emph{formal variables} $m_1,m_2, \dots, m_{n-2}$, and denote this polynomial by $G(m_1,m_2,\ldots,m_{n-2})$. Fix integers $a_1,\ldots,a_{n-2}\ge0$ with $\sum_{i=1}^{n-2}a_i\le n-4,$ and let $M=\prod_{i=1}^{n-2}(m_i)^{a_i}$. The coefficient of the monomial $M$ in $G(m_1, \ldots, m_{n-2})$ is equal to $$\sum_{\substack{S\subseteq\{1,\ldots, n-2\}\\S\supseteq\{i:a_i>0\}}}(-1)^{\abs{S}}F_{M}(\abs{S}),$$ where $$F_{M}(x)=\sum_{\substack{B\subseteq\{1,\ldots,n-4\}\\\abs{B}=\sum_ia_i}}\binom{\abs{B}}{a_1,\ldots,a_{n-2}}\prod_{k\in\{1,\ldots,n-4\}\setminus B}\ell_k(x).$$
     
Then $F_M(x)$ has degree $$\deg(F_{M}(x))=n-4-\sum_{i=1}^na_i<n-2-\abs{\{i:a_i>0\}},$$ 
By Lemma \ref{lem:AlternatingSumLemma}, the coefficient of $M$ in $G(m_1, \ldots, m_{n-2})$ is zero.  Since $M$ was arbitrary, $G$ vanishes identically as desired.
\end{proof}

\begin{proof}[Proof of Theorem \ref{thm:ClosedFormula}]
It suffices to show that $\widetilde w_r(\vec m)$ satisfies conditions \eqref{it:3PointCondition}, \eqref{it:VanishingCondition}, and \eqref{it:recursion} of Lemma~\ref{lem:Reconstruct}. 

    If $n=3$, then for any $\vec m$, the only subsets $S$ satisfying \eqref{eq:ConditionStar} are those with $\abs{S}\le1$. Then $\widetilde w_r(\vec m)=\tfrac{3}{2} - \tfrac{1}{2} = 1$ as desired, proving condition \eqref{it:3PointCondition}.
    
    Next, suppose $\vec m$ is a numerical $r$-spin monodromy vector of length $n\ge 4$ with $m_i = r$ for some $i\in [n]$. Suppose without loss of generality that $i=n$. Let $S\subseteq[n-1]$. Observe that $S$ satisfies \eqref{eq:ConditionStar} if and only if $S\cup\{n\}$ satisfies \eqref{eq:ConditionStar}, and the summands corresponding to $S$ and $S\cup\{n\}$ of $\widetilde w_r(\vec m)$ are identical with opposite signs since $\abs{S\cup\{n\}}=\abs{S}+1$. This shows $\widetilde w_r(\vec m)=0$, proving \eqref{it:VanishingCondition}.

    Finally, we must prove that $\widetilde w_r$ satisfies \eqref{eq:Recursion}. We do so by induction on $n$, with base case $n=3$. When $n=3$, the calculation $\widetilde w_r(\vec m)=1$ above implies that both sides of \eqref{eq:Recursion} vanish. For $n\ge4$, let \begin{align}\label{eq:MStarDef}
        \vec m^*=(m_1^*,\ldots,m_n^*)=(m_1-1,m_2+1,m_3,\ldots,m_n).
    \end{align}

    Then we compute
    \begin{align}\label{eq:ClosedFormulaLeftSideOfRecursion}
        \widetilde w_r(\vec m)-\widetilde w_r(\vec m^*)&=\frac{n-3}{2r^{n-3}}\Biggl(\sum_{\substack{S\subseteq[n]\\
        \eqref{eq:ConditionStar}\\
        %\sum_{i\in S}m_i\ge(\abs{S}-1)r+n-2\\
        1\in S,\thickspace2\not\in S}}(-1)^{1+\abs{S}}\prod_{k=2}^{n-3}\Bigl(\Bigl(\sum_{i\in S}m_i\Bigr)-(\abs{S}-1)r-k\Bigr)\\
        &\quad-\sum_{\substack{S\subseteq[n]\\
        \eqref{eq:ConditionStar2}\\
        %\sum_{i\in S}m_i\ge(\abs{S}-1)r+n-3\\
        2\in S,\thickspace1\not\in S}}(-1)^{1+\abs{S}}\prod_{k=1}^{n-4}\Bigl(\Bigl(\sum_{i\in S}m_i\Bigr)-(\abs{S}-1)r-k\Bigr)\Biggr).\nonumber
    \end{align}
    The factor $n-3$ appears as the difference of a $k=1$ factor and a $k=n-3$ factor, from the definition  of $\widetilde{w}_r$. By induction, we have
\begin{align}\label{eq:ClosedFormulaRightSideOfRecursion}
        T_r^{1,2,3}(\vec m)&=\frac{1}{r}\Bigl(\delta_{m_2+m_3\ge r+1}\widetilde w_r(m_1-1,m_4,\ldots,m_{n},m_2+m_3-r)\\
    &\quad\quad-\delta_{m_1+m_3\ge r+2}\widetilde w_r(m_2,m_4,\ldots,m_{n},m_1+m_3-r-1)\Bigr).\nonumber
    \end{align}

    Let $\vec m^{**}=(m_1-1,m_4,\ldots,m_n,m_2+m_3-r)$, where we index this tuple by the set $\{1,4,\ldots,n,\dagger\}$ and let $\vec m^{***}=(m_2,m_4,\ldots,m_n,m_1+m_3-r-1)$, where we index this tuple by the set $\{2,4,\ldots,n,\bullet\}$. Plugging in \eqref{eq:WTildeDef} to \eqref{eq:ClosedFormulaRightSideOfRecursion}, we have four cases, according to whether $m_2+m_3\ge r+1$ and whether $m_1+m_3\ge r+2$. 

    \medskip
    
    \noindent\textbf{Case 1: $m_2+m_3\ge r+1$ and $m_1+m_3\ge r+2$.} By \eqref{eq:ClosedFormulaRightSideOfRecursion} and the definition of $\widetilde w_r$, we have
    \begin{align}\label{eq:T123ClosedFormula}
        T_r^{1,2,3}(\vec m)&=\frac{1}{2r^{n-3}}\Biggl(\sum_{\substack{S'\subseteq\{1,4,\ldots,n,\dagger\}\\
        \text{\eqref{eq:ConditionStar2} w.r.t. $\vec m^{**}$}\\
        %\sum_{i\in S'}m_i^{**}\ge(\abs{S'}-1)r+n-3
        }}(-1)^{1+\abs{S'}}\prod_{k=1}^{n-4}\Bigl(\Bigl(\sum_{i\in S'}m_i^{**}\Bigr)-(\abs{S'}-1)r-k)\\
        &\quad-\sum_{\substack{S''\subseteq\{2,4,\ldots,n,\bullet\}\\
        \text{\eqref{eq:ConditionStar2} w.r.t. $\vec m^{***}$}\\
        %\sum_{i\in S''}m_i^{***}\ge(\abs{S''}-1)r+n-3
        }}(-1)^{1+\abs{S''}}\prod_{k=1}^{n-4}\Bigl(\Bigl(\sum_{i\in S''}m_i^{***}\Bigr)-(\abs{S''}-1)r-k)\Biggr).\nonumber
    \end{align}
To clarify, e.g., in the first summation, we sum over all $S'$ that satisfy \eqref{eq:ConditionStar2} with respect to $\vec m^{**}$, as defined in Notation~\ref{star star star}.
    
    A subset $S'\subseteq\{4,\ldots,n\}$ (satisfying \eqref{eq:ConditionStar2} with respect to $\vec m^{**}$, and equivalently with respect to $\vec m^{***}$) contributes to both sums in \eqref{eq:T123ClosedFormula} with opposite signs. The contribution from a subset $S'\supseteq\{1,\dagger\}$ (satisfying \eqref{eq:ConditionStar2} with respect to $\vec m^{**}$) cancels with an identical contribution from the corresponding subset $S''=(S'\cup\{2,\bullet\})\setminus\{1,\dagger\}$ (which necessarily satisfies \eqref{eq:ConditionStar2} with respect to $\vec m^{***}$). 
    
    We now translate the sums over $S'$ and $S''$ in \eqref{eq:T123ClosedFormula} into sums over $S$. If a subset $S'\subseteq\{1,4,\ldots,n,\dagger\}$ contains $\dagger,$ then $S'$ satisfies \eqref{eq:ConditionStar2} with respect to $\vec m^{**}$ if and only if $S:=(S'\setminus\{\dagger\})\cup\{2,3\}$ satisfies \eqref{eq:ConditionStar2} with respect to $\vec m$. If instead $S'$ contains $1$, then $S'$ satisfies \eqref{eq:ConditionStar2} with respect to $\vec m^{**}$ if and only if $S:= S'$ satisfies \eqref{eq:ConditionStar} with respect to $\vec m$. %$$\sum_{i\in S'}m_i^{**}\ge(\abs{S'}-1)r+n-3\text{\quad\quad if and only if \quad\quad }\sum_{i\in S}m_i\ge(\abs{S}-1)r+n-3.$$ 
    Using these observations (and similar calculations for $S''$), we rewrite the sums in  \eqref{eq:T123ClosedFormula} as sums over subsets of $[n]$:
    \begin{align}\label{eq:T123ClosedFormula2}
        T_r^{1,2,3}(\vec m)&=\frac{1}{2r^{n-3}}\Biggl(\sum_{\substack{S\subseteq[n]\\
        \eqref{eq:ConditionStar}\\
        %\sum_{i\in S}m_i\ge(\abs{S}-1)r+n-2\\
        1\in S,\thickspace 2,3\not\in S}}(-1)^{1+\abs{S}}\prod_{k=2}^{n-3}\Bigl(\Bigl(\sum_{i\in S}m_i\Bigr)-(\abs{S}-1)r-k\Bigr)\\
        &+\sum_{\substack{S\subseteq[n]\\
        \eqref{eq:ConditionStar2}\\
        %\sum_{i\in S}m_i\ge(\abs{S}-1)r+n-3\\
        2,3\in S,\thickspace 1\not\in S}}(-1)^{\abs{S}}\prod_{k=1}^{n-4}\Bigl(\Bigl(\sum_{i\in S}m_i\Bigr)-(\abs{S}-1)r-k\Bigr)\nonumber\\
        &-\sum_{\substack{S\subseteq[n]\\
        \eqref{eq:ConditionStar2}\\
        %\sum_{i\in S}m_i\ge(\abs{S}-1)r+n-3\\
        2\in S,\thickspace 1,3\not\in S}}(-1)^{1+\abs{S}}\prod_{k=1}^{n-4}\Bigl(\Bigl(\sum_{i\in S}m_i\Bigr)-(\abs{S}-1)r-k\Bigr)\nonumber\\
        &-\sum_{\substack{S\subseteq[n]\\
        \eqref{eq:ConditionStar}\\
        %\sum_{i\in S}m_i\ge(\abs{S}-1)r+n-2\\
        1,3\in S,\thickspace 2\not\in S}}(-1)^{\abs{S}}\prod_{k=2}^{n-3}\Bigl(\Bigl(\sum_{i\in S}m_i\Bigr)-(\abs{S}-1)r-k\Bigr)\Biggr).\nonumber
    \end{align}
    Combining the first and last sums in \eqref{eq:T123ClosedFormula2}, as well as the second and third sums, and comparing with \eqref{eq:ClosedFormulaLeftSideOfRecursion}, shows that \eqref{eq:Recursion} is satisfied. This completes Case 1.

    \medskip

    \noindent\textbf{Case 2: $m_2+m_3\le r$ and $m_1+m_3\ge r+2$.} In this case, if a subset $S\subseteq[n]$ satisfies $2,3\not\in S$, then by an easy calculation, $S$ must satisfy \eqref{eq:ConditionStar} with respect to $\vec m$.
    %necessarily $\sum_{i\in S}m_i\ge(n-2)(r+1)-(r+(n-\abs{S}-2)r)=(\abs{S}-1)r+n-2,$ so $S$ satisfies \eqref{eq:ConditionStar}. 
    Similarly, if $2,3\in S$, then $S$ does not satisfy \eqref{eq:ConditionStar}. 
    %$$\sum_{i\in S}m_i=(n-2)(r+1)-\sum_{i\not\in S}m_i\ge(n-2)(r+1)-m_2-m_3-(n-2-\abs{S})r\ge(\abs{S}-1)r+n-2.$$ 
    Thus \eqref{eq:ClosedFormulaLeftSideOfRecursion} reads: \begin{align}\label{eq:ClosedFormulaLeftSideOfRecursionSecondCase}
        \widetilde w_r(\vec m)-\widetilde w_r(\vec m^*)&=\frac{n-3}{2r^{n-3}}\Biggl(\sum_{\substack{S\subseteq[n]\\
        \eqref{eq:ConditionStar}\\
        %\sum_{i\in S}m_i\ge(\abs{S}-1)r+n-2\\
        1,3\in S,\thickspace2\not\in S}}(-1)^{1+\abs{S}}\prod_{k=2}^{n-3}\Bigl(\Bigl(\sum_{i\in S}m_i\Bigr)-(\abs{S}-1)r-k\Bigr)\\
        &\quad+\sum_{\substack{S\subseteq[n]\\1\in S,\thickspace2,3\not\in S}}(-1)^{1+\abs{S}}\prod_{k=2}^{n-3}\Bigl(\Bigl(\sum_{i\in S}m_i\Bigr)-(\abs{S}-1)r-k\Bigr)\nonumber\\
        &\quad-\sum_{\substack{S\subseteq[n]\\
        \eqref{eq:ConditionStar2}\\
        %\sum_{i\in S}m_i\ge(\abs{S}-1)r+n-3\\
        2\in S,\thickspace1,3\not\in S}}(-1)^{1+\abs{S}}\prod_{k=1}^{n-4}\Bigl(\Bigl(\sum_{i\in S}m_i\Bigr)-(\abs{S}-1)r-k\Bigr)\Biggr).\nonumber
    \end{align}
    On the other hand, \eqref{eq:ClosedFormulaRightSideOfRecursion} reads:
    \begin{align}\label{eq:T123ClosedFormulaSecondCase}
        T_r^{1,2,3}(\vec m)&=-\frac{1}{r}\widetilde w_r(m_2,m_4,\ldots,m_n,m_1+m_3-r-1)\\
        &=\frac{1}{2r^{n-3}}\sum_{\substack{S''\subseteq\{2,4,\ldots,n,\bullet\}\\
        \text{\eqref{eq:ConditionStar} w.r.t. $\vec m^{***}$}\\
        %\sum_{i\in S''}m_i^{***}\ge(\abs{S''}-1)r+n-3
        }}(-1)^{1+\abs{S''}}\prod_{k=1}^{n-4}\Bigl(\Bigl(\sum_{i\in S''}m_i^{***}\Bigr)-(\abs{S''}-1)r-k\Bigr).\nonumber
    \end{align}
    If a subset $S''$ contains both 2 and $\bullet$, then by the assumption $m_2+m_3\le r$, $S''$ does not contribute to the sum. Using this, and rewriting all sums in terms of subsets of $[n]$ as above, we obtain:
    \begin{align}\label{eq:T123ClosedFormulaSecondCase2}
        T_r^{1,2,3}(\vec m)&=\frac{-1}{2r^{n-3}}\Biggl(\sum_{\substack{S\subseteq[n]\\
        \eqref{eq:ConditionStar2}\\
        %\sum_{i\in S}m_i\ge(\abs{S}-1)r+n-3\\
        2\in S,\thickspace1,3\not\in S}}(-1)^{1+\abs{S}}\prod_{k=1}^{n-4}\Bigl(\Bigl(\sum_{i\in S}m_i\Bigr)-(\abs{S}-1)r-k\Bigr)\\
        &+\sum_{\substack{S\subseteq[n]\\
        \eqref{eq:ConditionStar}\\
        %\sum_{i\in S}m_i\ge(\abs{S}-1)r+n-2\\
        1,3\in S,\thickspace2\not\in S}}(-1)^{\abs{S}}\prod_{k=2}^{n-3}\Bigl(\Bigl(\sum_{i\in S}m_i\Bigr)-(\abs{S}-1)r-k\Bigr)\nonumber\\
        &+\sum_{\substack{S\subseteq[n]\\1,2,3\not\in S}}(-1)^{1+\abs{S}}\prod_{k=1}^{n-4}\Bigl(\Bigl(\sum_{i\in S}m_i\Bigr)-(\abs{S}-1)r-k\Bigr)\Biggr).\nonumber
     \end{align}
     Note that the first two sums in \eqref{eq:T123ClosedFormulaSecondCase2} match the first and third sums in \eqref{eq:ClosedFormulaLeftSideOfRecursionSecondCase}. To prove that \eqref{eq:Recursion} holds, it remains to prove
     \begin{align}\label{eq:T123ClosedFormulaSecondCase3}
        \sum_{S\subseteq\{1,4,\ldots,n\}}(-1)^{1+\abs{S}}\prod_{k=2}^{n-3}\Bigl(\Bigl(\sum_{i\in S}m_i^*\Bigr)-(\abs{S}-1)r-k\Bigr)=0,
     \end{align}
     where $m_i^*$ is as in \eqref{eq:MStarDef}. Here we have written $m_i^*$ because the second sum in \eqref{eq:ClosedFormulaLeftSideOfRecursionSecondCase} and the third sum in \eqref{eq:T123ClosedFormulaSecondCase2} can only be combined cleanly after the substitution $m_1\mapsto m_1-1$. One sees that \eqref{eq:T123ClosedFormulaSecondCase3} is a special case of Lemma~\ref{lem:low degree vanishing}. This completes Case 2.

    \medskip

    \noindent\textbf{Case 3: $m_2+m_3\ge r+1$ and $m_1+m_3\le r+1$.} This case is almost identical to Case 2. Similar computations to those in Case 2 give \begin{align}\label{eq:ClosedFormulaLeftSideOfRecursionThirdCase}
        \widetilde w_r(\vec m)-\widetilde w_r(\vec m^*)&=\frac{n-3}{2r^{n-3}}\Biggl(\sum_{\substack{S\subseteq[n]\\
        \eqref{eq:ConditionStar}\\
        %\sum_{i\in S}m_i\ge(\abs{S}-1)r+n-2\\
        1\in S,\thickspace2,3\not\in S}}(-1)^{1+\abs{S}}\prod_{k=2}^{n-3}\Bigl(\Bigl(\sum_{i\in S}m_i\Bigr)-(\abs{S}-1)r-k\Bigr)\\
        &\quad-\sum_{\substack{S\subseteq[n]\\
        \eqref{eq:ConditionStar2}\\
        %\sum_{i\in S}m_i\ge(\abs{S}-1)r+n-3\\
        2,3\in S,\thickspace1\not\in S}}(-1)^{1+\abs{S}}\prod_{k=1}^{n-4}\Bigl(\Bigl(\sum_{i\in S}m_i\Bigr)-(\abs{S}-1)r-k\Bigr)\nonumber\\
        &\quad-\sum_{\substack{S\subseteq[n]\\2\in S,\thickspace1,3\not\in S}}(-1)^{1+\abs{S}}\prod_{k=1}^{n-4}\Bigl(\Bigl(\sum_{i\in S}m_i\Bigr)-(\abs{S}-1)r-k\Bigr)\Biggr).\nonumber
    \end{align}
    and
    \begin{align}\label{eq:T123ClosedFormulaThirdCase}
        T_r^{1,2,3}(\vec m)&=\frac{1}{2r^{n-3}}\Biggl(\sum_{\substack{S\subseteq[n]\\
        \eqref{eq:ConditionStar}\\
        %\sum_{i\in S}m_i\ge(\abs{S}-1)r+n-2\\
        1\in S,\thickspace2,3\not\in S}}(-1)^{1+\abs{S}}\prod_{k=2}^{n-3}\Bigl(\Bigl(\sum_{i\in S}m_i\Bigr)-(\abs{S}-1)r-k\Bigr)\\
        &+\sum_{\substack{S\subseteq[n]\\
        \eqref{eq:ConditionStar2}\\
        %\sum_{i\in S}m_i\ge(\abs{S}-1)r+n-3\\
        2,3\in S,\thickspace1\not\in S}}(-1)^{\abs{S}}\prod_{k=1}^{n-4}\Bigl(\Bigl(\sum_{i\in S}m_i\Bigr)-(\abs{S}-1)r-k\Bigr)\nonumber\\
        &+\sum_{\substack{S\subseteq[n]\\1,2,3\not\in S}}(-1)^{1+\abs{S}}\prod_{k=1}^{n-4}\Bigl(\Bigl(\sum_{i\in S}m_i\Bigr)-(\abs{S}-1)r-k\Bigr)\Biggr).\nonumber
     \end{align} 
    From  \eqref{eq:ClosedFormulaLeftSideOfRecursionThirdCase} and \eqref{eq:T123ClosedFormulaThirdCase}, we see that \eqref{eq:Recursion} is equivalent to $$\sum_{S\subseteq\{2,4,\ldots,n\}}(-1)^{1+\abs{S}}\prod_{k=1}^{n-4}\Bigl(\Bigl(\sum_{i\in S}m_i\Bigr)-(\abs{S}-1)r-k\Bigr)=0,$$ which is a special case of Lemma \ref{lem:low degree vanishing}. This completes Case 3.

    \medskip
    
    \noindent\textbf{Case 4: $m_2+m_3\le r$ and $m_1+m_3\le r+1$.} In this case, if $S\subseteq [n]$ satisfies $1,3\not\in S$ or $2,3\not\in S,$ then $S$ must satisfy \eqref{eq:ConditionStar}. Similarly, if $2,3\in S$ or $1,3\in S$, then $S$ does not satisfy \eqref{eq:ConditionStar}. Thus \eqref{eq:ClosedFormulaLeftSideOfRecursion} simplifies to:
    \begin{align}\label{eq:ClosedFormulaLeftSideOfRecursionFourthCase}
        \widetilde w_r(\vec m)-\widetilde w_r(\vec m^*)&=\frac{n-3}{2r^{n-3}}\Biggl(\sum_{\substack{S\subseteq\{1,4,\ldots,n\}}}(-1)^{1+\abs{S}}\prod_{k=2}^{n-3}\Bigl(\Bigl(\sum_{i\in S}m_i\Bigr)-(\abs{S}-1)r-k\Bigr)\\
        &\quad-\sum_{\substack{S\subseteq\{2,4,\ldots,n\}}}(-1)^{1+\abs{S}}\prod_{k=1}^{n-4}\Bigl(\Bigl(\sum_{i\in S}m_i\Bigr)-(\abs{S}-1)r-k\Bigr)\Biggr),\nonumber
    \end{align}
    while \eqref{eq:ClosedFormulaRightSideOfRecursion} is equal to zero. Again, Lemma \ref{lem:low degree vanishing} implies that both sums in \eqref{eq:ClosedFormulaLeftSideOfRecursionFourthCase} are equal to zero, so \eqref{eq:Recursion} is satisfied, completing Case 4.

    We have now checked that $\widetilde w_r$ satisfies all of the conditions of Corollary \ref{lem:Reconstruct}, so we conclude $\widetilde w_r=w_r$.
\end{proof}

\section{Monotonicity of genus $0$ $r$-spin invariants}\label{sec:monotonicity}
The techniques introduced also imply that genus $0$ $r$-spin invariants satisfy a monotonicity property which is not obvious from the closed formula in the previous section. 

Recall the \emph{dominance} partial ordering on partitions of an integer $N$, where for partitions $\vec p=(p_1,p_2,\ldots)$ and $\vec p\thinspace'=(p_1',p_2',\ldots)$ of $N$ with $p_1\le p_2\le\cdots$ and $p_1'\le p_2'\le\cdots$, we say $\vec p\le\vec p\thinspace'$ if for all $i\ge0$ we have $p_1+\cdots+p_i\le p_1'+\cdots+p_i'$. That is, $\vec p\le\vec p\thinspace'$ if $\vec p$ is a ``more balanced distribution of $N$'' than $\vec p\thinspace'$. If $\vec p \le \vec p\thinspace'$, we say $\vec p\thinspace'$ \emph{dominates} $\vec p$.
\begin{remark}
    The dominance partial ordering is a fundamental structure of partitions (or more generally vectors of real numbers with fixed sum), and thus arises in many contexts inside (and outside) mathematics. See \cite{MarshallOlkinArnold} for a survey of such appearances. In algebraic geometry, the dominance partial ordering is known for (among other things) its connections to the geometry of Hilbert schemes of points on surfaces, see \cite{Nakajima}.
\end{remark}

\begin{thm}\label{thm:RespectsOrdering} Genus-zero $r$-spin invariants satisfy the following properties:
\begin{enumerate}
    \item For any $r$ and $n$, $w_r$ is a weakly order-reversing function on $n$-part partitions of $(n-2)(r+1)$ with respect to the dominance order. That is, if $\vec m$ and $\vec m'$ are two numerical $r$-spin monodromy vectors and $\vec m\le\vec m'$, then $w_r(\vec m)\ge w_r(\vec m')$.\label{it:Ordering}
    \item For any numerical $r$-spin monodromy vector $\vec m$, $w_r(\vec m)\ge0$, with $w_r(\vec m)\ne0$ if and only if $n-2\le m_i\le r-1$ for all $i$.\label{it:Positivity}
\end{enumerate}
\end{thm}

We prove the theorem in several steps, beginning with the following standard fact about the dominance order (see \cite[Prop. 2.3]{Brylawski}).
\begin{fact}\label{obs:GiveCookies}
    A partition $\vec p$ is dominated by $\vec p\thinspace'$ if and only if there is a finite sequence
$$\vec p\thinspace'=\vec p\thinspace^0,\vec p\thinspace^1,\ldots,\vec p\thinspace^k=\vec p,$$
so that $\vec p\thinspace^i$ is obtained from $\vec p\thinspace^{i-1}$ by replacing two parts $a,b\in\vec p\thinspace^{i-1}$, where $a<b-1$, with nonnegative integers $c =a+1$ and $d=b-1$, respectively. That is, $\vec p\thinspace^i$ and $\vec p\thinspace^{i+1}$ are neighbors and $\vec p\thinspace^{i+1}$ dominates $\vec p\thinspace^i$. Intuitively, a `richer' part ($b$) donates one `cookie' to a `poorer' part ($a$), but only if doing so would not make $b$ poorer than $a$.
\end{fact}

\begin{lemma}\label{lem:cookieOrder}
Suppose $\vec m$ and $\vec m'$ are numerical $r$-spin monodromy vectors. If $\vec m$ and $\vec m'$ are neighbors and $\vec m \ge \vec m'$, then $w_r(\vec m) \le w_r( \vec m')$.
\end{lemma}
\begin{proof}
    We proceed by induction on $n$. The base case is $n=4$, which follows from Proposition \ref{prop:34Point}.

Fix $n$, and assume both statements hold for $k$-pointed monodromy vectors if $k<n$. Without loss of generality, write  $\vec{m}=(m_1,\ldots,m_n)$ and $\vec m' = (m_1 - 1, m_2 +1, m_3, \ldots, m_n)$ with $1<m_1\le m_2<r$ as $\vec m \le \vec m'$.  For $j\in\{4,\ldots,n\},$ we have by definition $$T_r^{1,2,j}(m_1,\ldots,m_n)=w_r(\vec m^1)w_r(\vec m^2)-w_r(\vec m^3)w_r(\vec m^4),$$ where:
\begin{align*}
    \vec m^1&=(m_1-1,m_3,\ldots,\hat{m_j},\ldots,m_{n},m_2+m_j-r)&\vec m^2&=(2,m_2,m_j,2r-m_2-m_j)\\
    \vec m^3&=(m_2,m_3,\ldots,\hat{m_j},\ldots,m_{n},m_1+m_j-r-1)&\vec m^4&=(2,m_1-1,m_j,2r+1-m_1-m_j).
\end{align*}

Note first that $$
    (m_1-1)+(m_2+m_j-r)=(m_1+m_j-r-1)+(m_2)$$
    and $$m_2>\max\{m_1-1,m_2+m_j-r\} \ge \min\{m_1-1,m_2+m_j-r\} > m_1 +m_j -r - 1.$$
 Thus there is a finite sequence like that in Fact~\ref{obs:GiveCookies} between $\vec m^3$ and $\vec m^1$.
This proves $\vec m^1\le\vec m^3$, so by the inductive hypothesis, we have 
\begin{align}\label{eq:m1m3}
    w_r(\vec m^1)\ge w_r(\vec m^3).
\end{align}

Next, we prove $w_r(\vec m^2)\ge w_r(\vec m^4)$. As both are 4-point invariants with an insertion of 2, both are either $0$ or $\tfrac{1}{r}$. We need only to show that if $w_r(\vec m^2) = 0$ then $w_r(\vec m^4) = 0$. % rule out $(W_r(\vec m^2),W_r(\vec m^4))=(0,1/r).$ 
Indeed, using Proposition~\ref{prop:34Point} and the fact that $m_2 < r$,  $w_r(\vec m^2)=0$ implies that either $2r-m_2-m_j\ge r$ or $m_j = 1$, which in turn implies either $2r+1-m_1-m_j\ge r$ or $m_j = 1$. Thus $w_r(\vec m^4)=0$. In conclusion,
\begin{align}\label{eq:m2m4}
    w_r(\vec m^2)\ge w_r(\vec m^4).
\end{align}

By the inductive hypothesis, the expressions $w_r(\vec m^1),$ $w_r(\vec m^2),$ $w_r(\vec m^3),$ and $w_r(\vec m^4)$ are nonnegative, so \eqref{eq:m1m3} and \eqref{eq:m2m4} imply $T_r^{1,2,j}\ge0$.
Thus 
\begin{align}\label{eq:InequalityOneStep}
    w_r(m_1,\ldots,m_n)&\ge w_r(m_1-1,m_2+1,m_3,\ldots,m_n).\qedhere
\end{align}
\end{proof}

\begin{lemma}\label{lem:SmallestInvariant}
    Let $\vec m=(m_1,\ldots,m_n)$ be a monodromy vector with $m_1=n-2$ and $n-2\le m_i\le r-1$ for all $i$. Then $w_r(\vec m)=\frac{(n-3)!}{r^{n-3}}$.
\end{lemma}
\begin{proof}
    We proceed by induction on $n,$ with base case $n=4$ following from Proposition \ref{prop:34Point}. 

    For $n>4,$ we apply Theorem \ref{thm:Recursion} to $\vec m$, with $(i,j,k)=(1,2,3)$. Since $m_1=n-2,$ by Corollary \ref{cor:RSpinVanishing}, $w_r(m_1-1,m_2+1,m_3,\ldots,m_n)=0$, so Theorem \ref{thm:Recursion} reads $$w_r(\vec m)=(n-3)T^{1,2,3}_r(\vec m).$$ We have $$m_2+m_3-r=(n-2)(r+1)-(n-2)-r-\sum_{i=4}^nm_i\ge(n-2)(r+1)-(n-2)-r-(n-3)(r-1)=n-3,$$ and $$m_1+m_3-r-1\le (n-2)+(r-1)-r-1=n-4,$$ which implies $$T^{1,2,3}_r(\vec m)=\frac{1}{r}(1\cdot w_r(m_1-1,m_4,\ldots,m_n,m_2+m_3-r)-0)=\frac{1}{r}\cdot\frac{(n-4)!}{r^{n-4}}$$ by induction. Thus $w_r(\vec m)=(n-3)T^{1,2,3}_r(\vec m)=\frac{(n-3)!}{r^{n-3}}.$
\end{proof}

\begin{proof}[Proof of Theorem~\ref{thm:RespectsOrdering}]
Fix an $n$-pointed $r$-spin monodromy vector $\vec m'$ with $\vec m'\ge\vec m$. By Fact \ref{obs:GiveCookies}, there is a finite sequence $\vec m'=\vec m\thinspace^0>\vec m\thinspace^1>\cdots>\vec m\thinspace^k=\vec m$ of neighbors, where $\vec m\thinspace^{i}$ is dominated by $\vec m\thinspace^{i+1}$. By Lemma \ref{lem:cookieOrder}, $w_r(\vec m\thinspace^{i})\le w_r(\vec m\thinspace^{i+1})$ for all $i$, so we conclude $W_r(\vec m)\ge W_r(\vec m')$. This proves statement \eqref{it:Ordering}.

Applying Lemma \ref{lem:cookieOrder} repeatedly,
we have $$w_r(\vec m)-w_r(m_1-k,m_2+k,m_3,\ldots,m_n)\ge0$$ for any $k\le\min\{m_1-1,r-m_2\}$. Taking $k=\min\{m_1-1,r-m_2\}$ gives $w_r(\vec m)\ge0$. This proves the first half of statement \eqref{it:Positivity}.

For the second half of statement \eqref{it:Positivity}, the `if' direction follows from Proposition \ref{prop:RamondVanishing} and Corollary \ref{cor:RSpinVanishing}. For the `only if' direction, suppose $\vec m=(m_1,\ldots,m_n)$ is a monodromy vector such that $n-2\le m_i\le r-1$ for all $i$, with $m_1\le\cdots\le m_n$. It is easy to check using Fact \ref{obs:GiveCookies} that $\vec m$ is dominated by a monodromy vector $\vec m'$ with $m_1=n-2$ and $n-2\le m_i\le r-1$ for $i\in\{2,\ldots,n\}$. (Roughly, $m_1$ repeatedly `donates a cookie' to a part $m_i<r-1$; such a part is guaranteed to exist by $\sum_im_i=(n-2)(r+1),$ and $\vec m'\ge\vec m$ follows from the assumption that $m_1\le m_2\le\cdots\le m_n$.) Thus by statement \eqref{it:Ordering} proved just above, it is sufficient to prove $w_r(\vec m)>0$ in the case where $m_1=n-2$ and $n-2\le m_i\le r-1$ for all $i$. This is immediate from Lemma \ref{lem:SmallestInvariant}.
\end{proof}

\bibliography{GZeroRevised.bbl}
\bibliographystyle{amsalpha}
\end{document}